\documentclass[12pt,a4paper]{article}
\usepackage{amssymb, amsmath}
\usepackage{mathtools}
\usepackage{hyperref}
\usepackage{tikz}
\usetikzlibrary{matrix,arrows}
\usepackage{tikz-cd}

\usepackage{todonotes}

\newcommand{\ot}{\otimes}

\newcommand{\otexp}[2]{{#1^{\ot #2}}}

\makeatletter
\newcommand{\catVec}{\@ifnextchar[{\bracketCVec}{\nobracketCVec}}
\def\bracketCVec[#1]{ \mathsf{Vec}_{#1} }
\def\nobracketCVec{ \mathsf{Vec} }
\makeatother

\makeatletter
\newcommand{\downar}{\@ifnextchar*{\starDownAr}{\nostarDownAr}}
\def\starDownAr*{ {\raisebox{0.5pt}[0pt][0pt]{\ensuremath{\scriptstyle{\downarrow}}}} }
\def\nostarDownAr{ {\raisebox{1pt}[0pt][0pt]{\ensuremath{\downarrow}}} }
\makeatother

\makeatletter
\newcommand{\upar}{\@ifnextchar*{\starUpAr}{\nostarUpAr}}
\def\starUpAr*{ {\raisebox{1pt}[0pt][0pt]{\ensuremath{\scriptstyle{\uparrow}}}} }
\def\nostarUpAr{ {\raisebox{1pt}[0pt][0pt]{\ensuremath{\uparrow}}} }
\makeatother

\makeatletter
\newcommand{\id}{\@ifnextchar[{\tensorID}{\notensorID}}
\def\tensorID[#1]{ \otexp{\mathsf{1}}{#1} }
\def\notensorID{ \mathsf{1}	}
\makeatother

\newcommand{\set}[2]{\left\{ #1 \ | \ #2 \right\} }

\newcommand{\dd}{\partial}	

\newcommand{\dg}[1]{{\left| #1 \right|}}

\newcommand{\into}{\hookrightarrow}

\newcommand{\onto}{\twoheadrightarrow}

\newcommand{\N}{\mathbb{N}}

\newcommand{\Z}{\mathbb{Z}}

\newcommand{\ol}[1]{\overline{#1}}

\newcommand{\op}{\oplus}

\newcommand{\pr}{\mathop{\mathrm{pr}}\nolimits}

\newcommand{\Ker}{\mathop{\mathrm{Ker}}\nolimits}

\newcommand{\colim}{\mathop{\mathrm{colim}}}

\newcommand{\Hom}{\mathop{\mathrm{Hom}}\nolimits}

\makeatletter
\newcommand{\quism}{\@ifnextchar[{\quismName}{\quismNoName}}
\def\quismName[#1]{\xrightarrow[#1]{\sim}}
\def\quismNoName{\xrightarrow{\sim}}
\makeatother

\newcommand{\To}{\xrightarrow}

\newcommand{\USh}[1]{\mathrm{USh}(#1)}

\newcommand{\sgn}[1]{\mathop{\mathrm{sgn}}\left(#1\right)}

\newcommand{\oC}{\mathcal{C}}

\newcommand{\oE}{\mathcal{E}}
\newcommand{\oR}{\mathcal{R}}

\newcommand{\oAss}{\mathcal{A}\mathit{ss}}

\newcommand{\oLie}{\mathcal{L}\mathit{ie}}

\newcommand{\oCom}{\mathcal{C}\mathit{om}}

\newcommand{\oEnd}[1]{{\mathcal{E}\mathit{nd}_{#1}}}	

\makeatletter
\newcommand{\smod}
	{\@ifnextchar*{\smod@nonsigma}{\smod@sigma}}
\newcommand{\smod@nonsigma}[1]	
	{\@ifnextchar[{\smod@nonsigma@col}{\smod@nonsigma@uncol}}
\newcommand{\smod@sigma}
	{\@ifnextchar[{\smod@sigma@col}{\smod@sigma@uncol}}
\def\smod@sigma@col[#1]{{\mbox{\ensuremath{#1}-\ensuremath{\Sigma}-module}}}
\def\smod@nonsigma@col[#1]{{\mbox{\ensuremath{#1}-collection}}}
\def\smod@sigma@uncol{\mbox{\ensuremath{\Sigma}-module} }
\def\smod@nonsigma@uncol{{\mbox{collection}} }
\makeatother

\newcommand{\ar}[1]{{\mathrm{ar}(#1)}}

\makeatletter
\newcommand{\Fr}{\@ifnextchar[{\Fr@weight}{\Fr@noweight}}
\def\Fr@weight[#1]#2{\mathbb{F}^{#1}(#2)}
\def\Fr@noweight#1{\mathbb{F}(#1)}
\makeatother

\newcommand{\oo}{\circ}

\newcommand{\fpr}{\mathop{\mathrm{*}}\nolimits}	

\newcommand{\Kdual}[1]{#1^{\textrm{<}}}

\makeatletter
\newcommand{\TJ}{\@ifnextchar*{\TJ@star}{\TJ@nostar}}
\def\TJ@nostar#1{\textrm{TJ}(#1)}
\def\TJ@star*{\textrm{TJ}}
\makeatother

\newcommand{\fld}{\ensuremath{\Bbbk}}
\newcommand{\oP}{\mathcal{P}}	
\newcommand{\oDP}[1]{\mathcal{D}_k\mathcal{P}}	
\newcommand{\oDR}[1]{\mathcal{D}_k\mathcal{R}}
\newcommand{\ul}[1]{\underline{#1}}	
\newcommand{\cSmod}{\mathsf{\Sigma}}				
\newcommand{\cSmodft}{\mathsf{\Sigma_{f.t.}}}			
\newcommand{\Op}{\mathsf{Op}}
\newcommand{\Opft}{\mathsf{Op_{f.t.}}}
\newcommand{\CoOp}{\mathsf{CoOp}}
\newcommand{\CoOpft}{\mathsf{CoOp_{f.t.}}}
\newcommand{\oQ}{\mathcal{Q}}
\newcommand{\gr}{\textrm{gr}}

\newcommand{\btheta}{{\boldsymbol{\theta}}}
\newcommand{\bm}{{\boldsymbol{m}}}

\newcommand{\CoDer}{\mathop{\mathrm{CoDer}}\nolimits}
\newcommand{\oS}{\mathcal{S}}

\usepackage{amsthm}

\swapnumbers
\theoremstyle{definition}
\newtheorem{theorem}{Theorem}[section]
\newtheorem{proposition}[theorem]{Proposition}

\newtheorem{lemma}[theorem]{Lemma}

\newtheorem{example}[theorem]{Example}
\newtheorem{remark}[theorem]{Remark}
\newtheorem{definition}[theorem]{Definition}

\title{Homotopy Derivations}
\author{
Martin Doubek\thanks{The author was supported by GA\v{C}R P201/13/27340P.}, \medskip \\
\textit{Charles University,} \\
\textit{Faculty of Mathematics and Physics,}\\
\textit{Prague}\\
\texttt{martindoubek@seznam.cz}
\and
Tom Lada, \medskip \\
\textit{North Carolina State University} \\
\textit{Mathematics Dept}\\
\textit{Raleigh, NC}\\
\texttt{lada@math.ncsu.edu}
}
\date{}

\begin{document}

\maketitle
\abstract{
We define a strong homotopy derivation of (cohomological) degree $k$ of a strong homotopy algebra over an operad $\oP$.
This involves resolving the operad obtained from $\oP$ by adding a generator with ``derivation relations''.
For a wide class of Koszul operads $\oP$, in particular $\oAss$ and $\oLie$, we describe the strong homotopy derivations by coderivations and show that they are closed under the Lie bracket.
We show that symmetrization of a strong homotopy derivation of an $A_\infty$ algebra yields a strong homotopy derivation of the symmetrized $L_\infty$ algebra.
We give examples of strong homotopy derivations generalizing inner derivations.
}

\tableofcontents

\section{Introduction}

An early account of degree $k$ derivations for graded associative and graded Lie algebras by Gerstenhaber may be found in \cite{G}.
Kajiura and Stasheff \cite{KS} defined strong homotopy derivations of degree $+1$ for $A_\infty$ algebras and Tolley \cite{T} did the same for $L_\infty $ algebras.  In this note, we study degree $k$ strong homotopy derivations of $A_\infty$ and $L_\infty$ algebras.  This will take place in the more general context of strong homotopy derivations of strong homotopy $\mathcal{P}$ algebras (a.k.a. $\oP_\infty$ algebras) where $\mathcal{P}$ is an arbitrary operad.

The notion of strong homotopy derivation can be easily defined in terms of resolutions of operads.
Surprisingly, if the defining relations of the $\oP_\infty$ algebra are known explicitly, the defining relations of the strong homotopy derivation can be made completely explicit too.
We emphasize that Koszulness of $\oP$ doesn't play any role here.
The resolution has been constructed in \cite{IB} and later developed in \cite{GS}, although for a different purpose than to make strong homotopy derivations explicit.
We need to do small technical adaptations of this construction.
This is contained in Section \ref{SECGT}, and there the reader is assumed to be familiar with basic notions of operad theory as in \cite{MSS} and \cite{LV}.

In Section \ref{SECAoo}, we make these derivations explicit for $A_\infty$ algebras by choosing the operad $\mathcal{P} $ to be $\mathcal{A}ss$, the operad for associative algebras.  We then define and construct \emph{inner} such derivations for $A_\infty$ algebras.

We then specialize the operad $\mathcal{P}$ to $\mathcal{L}ie$, the operad for Lie algebra, in Section \ref{SECLoo}.  We then arrive at our definition of strong homotopy derivations of $L_\infty$ algebras.  \emph{Inner} derivations of these algebras are then constructed.

Section \ref{SECSymComp} shows that the usual symmetrization of $A_\infty$ data yields $L_\infty$ data.  Specifically, we show that the symmetrization of a strong homotopy derivation of an $A_\infty$ algebra gives us a strong homotopy derivation of the $L_\infty$ algebra obtained by symmetrization of the $A_\infty$ algebra structure.  We also define the composition of these derivations by translating the data to the coalgebra level.

In Appendix, we discuss the case $\oP$ being Koszul.
Under some finiteness assumptions, we generalize the coderivation description from the $A_\infty$ and $L_\infty$ cases.
On the way, we reprove some earlier results of the paper by a direct application of the Koszul theory.

We thank Jim Stasheff and Martin Markl for helpful comments on earlier versions of the paper.




\section{General theory} \label{SECGT}

The plan of the section is as follows:
Given an operad $\oP$, we define an operad $\oDP{k}$ such that $\oDP{k}$ algebras are $\oP$ algebras together with a degree $k$ derivation of the algebra structure.
Given a resolution $\oR$ of $\oP$, we explicitly construct a resolution $\oDR{k}$ of $\oDP{k}$.
We then define strong homotopy derivation of degree $k$ to be a $\oDR{k}$ algebra, or rather a certain infinite set of operations in the $\oDR{k}$ algebra.

We recall the construction of \cite{IB} and \cite{GS} and briefly indicate changes needed for our purposes.
We closely follow \cite{GS}, where we wish to
\begin{enumerate}
\item specialize from many to single color,
\item generalize from $0$ to arbitrary degree of the derivation,
\item generalize from non-$\Sigma$ (a.k.a. non-symmetric) to $\Sigma$ (symmetric) operads.
\end{enumerate}
We also switch from homological to cohomological grading.
The results are stated in a self-contained way.

All dg vector spaces live over a fixed field $\fld$ of characteristics $0$; the differential has degree $+1$.
By $\Sigma$ module we mean a $\Sigma$ module in the category of dg vector spaces with degree $0$ dg maps.
Operad means a $\Sigma$ operad in the category of dg vector spaces.
Let $\oo_i:\oP(m)\ot\oP(n)\to\oP(m+n-1)$ be the structure operation of an operad $\oP$.
In particular, for a dg vector space $A$, the endomorphism operad $\oP=\oEnd{A}$ and $f:A^{\ot m}\to A,\ g:A^{\ot n}\to A$, we define
$$f\oo_i g:=f(\id^{\ot i-1}\ot g\ot\id^{\ot m-i}).$$

\begin{definition} \label{DEFDerivation}
Let $\oP$ be an operad and $A$ a dg vector space with differential $d$.
A $\oP$ algebra $A$ is given by an operad morphism $\alpha:\oP\to\oEnd{A}$.
Let $k\in\Z$.
A \emph{degree $k$ derivation of $A$} is a degree $k$ linear map $\theta:A\to A$ such that
\begin{gather*}
d\oo_1\theta=(-1)^k\theta\oo_1 d, \\
\theta\oo_1\alpha(p)=(-1)^{k\dg{p}}\sum_{i=1}^n \alpha(p)\oo_i\theta
\end{gather*}
for arbitrary $p\in\oP(n)$.
Evaluated on elements $a_1,\ldots,a_n\in A$, the second equation reads
$$\theta(\alpha(p)(a_1,\ldots,a_n))=(-1)^{k\dg{p}}\sum_{i=1}^n (-1)^{k(\dg{a_1}+\cdots+\dg{a_{i-1}})}\alpha(p)(a_1,\ldots,a_{i-1},\theta(a_i),a_{i+1},\ldots,a_n).$$
\end{definition}

Denote $\Fr{X}$ the free operad generated by a $\Sigma$ module $X$.
Denote $\fpr$ the coproduct in the category of operads, a.k.a. free product of operads.
Let $\oP$ and $\oQ$ be operads.
Recall that an operadic derivation $D:\oP\to\oQ$ is a morphism of the underlying $\Sigma$ modules satisfying $D\oo_i=\oo_i(D\ot\id+\id\ot D)$ for any $i$.
Derivations $\Fr{X}\to\oQ$ are in bijection with $\Sigma$ module morphisms $X\to\oQ$ via the restriction.
Similarly, derivations $\oP'\fpr\oP''\to\oQ$ are in bijections with pairs of derivations $\oP'\to\oQ$ and $\oP''\to\oQ$.

\begin{definition} \label{DEFDkP}
Let $\oP$ be an operad with a differential $\dd_\oP$.
Consider a $\Sigma$ module $\Phi := \fld\{\phi\}$, the $\fld$ linear span of the set $\{\phi\}$, such that $\phi$ is of arity 1 and degree $k$.
Let $\mathfrak{D}$ be the ideal in $\oP \fpr \Fr{\Phi}$ generated by all elements
\begin{gather*} 
\phi \oo_1 p - (-1)^{k\dg{p}}\sum_{i=1}^{n} p\oo_i\phi
\end{gather*}
for $p\in\oP(n)$.
Denote $$\oDP{k} := \left( \frac{\oP \fpr \Fr{\Phi}}{\mathfrak{D}} , \dd_{\oDP{k}} \right),$$
where $\dd_{\oDP{k}}$ is the degree $1$ operadic derivation given by the formulas
$$\dd_{\oDP{k}}(p) := \dd_\oP(p),\quad \dd_{\oDP{k}}(\phi) := 0$$
for $p\in\oP$.
\end{definition}

\begin{proposition}
A $\oDP{k}$ algebra on a dg vector space $A$ is a pair $(A,\theta)$, where $A$ is an $\oP$ algebra and $\theta$ is a derivation of the $\oP$ algebra $A$.
\end{proposition}

\begin{proof}
$\theta$ is the image of $\phi$ under $\oDP{k}\to\oEnd{A}$.
\end{proof}

For $m\in\Z$, let $\upar^m$ be the $m$-fold suspension functor of dg vector spaces or $\Sigma$ modules:
$(\upar^m X)^n:=X^{n-m}$ and $\upar^m$ also denotes the degree $m$ map $X\to\upar^m X$.
Recall that an operad of the form $(\Fr{X},\dd)$, where the differential $\partial$ needn't come from the free operad construction, is called quasi-free.

\begin{definition} \label{DEFoDR}
Let
\begin{gather*} 
\oR := (\Fr{X},\dd_\oR) \To{\rho_\oR} (\oP,\dd_{\oP}),
\end{gather*}
be a quasi-free resolution, where $X$ is a $\Sigma$ module.
Consider the free operad
$$\oDR{k} := \Fr{X\op\Phi\op\ul{X}},$$
where 
$$\underline{X}:=\upar^{k-1} X.$$
We denote by $\ul{x}$ the element $\upar^{k-1} x\in\ul{X}$ corresponding to $x\in X$.
To describe the differential, let $s:\Fr{X} \to \oDR{k}$ be a degree $k-1$ operadic derivation determined by
$$s(x) := \ul{x}\in\ul{X} \mbox{ for }x\in X.$$
Then define a degree $1$ derivation $\dd_{\oDR{k}} : \oDR{k} \to \oDR{k}$ by
\begin{align}
&\dd_{\oDR{k}}(x) := \dd_\oR(x), \nonumber \\
&\dd_{\oDR{k}}(\phi) := 0, \nonumber \\ 
&\dd_{\oDR{k}}(\ul{x}) := \phi \oo_1 x -(-1)^{k\dg{x}} \sum_{i=1}^n x\oo_i \phi \ -(-1)^k s(\dd_\oR(x)). \label{DiffOnoDR} 
\end{align}
\end{definition}

Let's clarify the last formula by an example:
Assume $\dd_\oR(x) = (x_1\oo_1 x_2)\oo_4 x_3+\cdots$.
In the standard pictorial notation, this is
$$\dd_\oR
\begin{tikzpicture}[baseline=-\the\dimexpr\fontdimen22\textfont2\relax] 
\draw (0,.5)--(0,-.5);
\draw (0,0)--(-.25,-.5);
\draw (0,0)--(-.5,-.5);
\draw (0,0)--(.5,-.5);
\draw (0,0)--(.25,-.5);
\filldraw (0,0) circle (3pt);
\node at (0,0) [right]{$\scriptstyle{x}$};
\end{tikzpicture}
=
\begin{tikzpicture}[baseline=-\the\dimexpr\fontdimen22\textfont2\relax] 
\draw (0,1.5)--(0,.5);
\draw (0,1)--(-.5,.5)--(-.5,0)--(-1,-.5);
\draw (-.5,0)--(0,-.5);
\draw (0,1)--(.5,.5)--(.5,-1)--(0,-1.5);
\draw (.5,-1)--(1,-1.5);
\filldraw (0,1) circle (3pt);
\node at (0,1) [left]{$\scriptstyle{x_1}$};
\filldraw (-.5,0) circle (3pt);
\node at (-.5,0) [left]{$\scriptstyle{x_2}$};
\filldraw (.5,-1) circle (3pt);
\node at (.5,-1) [right]{$\scriptstyle{x_3}$};
\end{tikzpicture}
+\cdots.
$$
Then
\begin{gather*}
\dd_\oR
\begin{tikzpicture}[baseline=-\the\dimexpr\fontdimen22\textfont2\relax] 
\draw (0,.5)--(0,-.5);
\draw (0,0)--(-.25,-.5);
\draw (0,0)--(-.5,-.5);
\draw (0,0)--(.5,-.5);
\draw (0,0)--(.25,-.5);
\filldraw[fill=white,draw=black] (0,0) circle (3pt);
\node at (0,0) [right]{$\scriptstyle{\ul{x}}$};
\end{tikzpicture}
=
\begin{tikzpicture}[baseline=-\the\dimexpr\fontdimen22\textfont2\relax] 
\draw (0,1)--(0,0);
\draw (0,0)--(0,-1);
\draw (0,-.5)--(-.25,-1);
\draw (0,-.5)--(-.5,-1);
\draw (0,-.5)--(.5,-1);
\draw (0,-.5)--(.25,-1);
\filldraw[fill=white,draw=black] (0,.5) circle (3pt);
\node at (0,.5) [right]{$\scriptstyle{\phi}$};
\filldraw (0,-.5) circle (3pt);
\node at (0,-.5) [right]{$\scriptstyle{x}$};
\end{tikzpicture}
-(-1)^{k\dg{x}} \bigg(
\begin{tikzpicture}[baseline=-\the\dimexpr\fontdimen22\textfont2\relax] 
\draw (-.5,0)--(-.5,-1);
\draw (0,1)--(0,0);
\draw (0,.5)--(-.25,0);
\draw (0,.5)--(-.5,0);
\draw (0,.5)--(.5,0);
\draw (0,.5)--(.25,0);
\filldraw[fill=white,draw=black] (-.5,-.5) circle (3pt);
\node at (-.5,-.5) [right]{$\scriptstyle{\phi}$};
\filldraw (0,.5) circle (3pt);
\node at (0,.5) [right]{$\scriptstyle{x}$};
\end{tikzpicture}
+
\begin{tikzpicture}[baseline=-\the\dimexpr\fontdimen22\textfont2\relax] 
\draw (-.25,0)--(-.25,-1);
\draw (0,1)--(0,0);
\draw (0,.5)--(-.25,0);
\draw (0,.5)--(-.5,0);
\draw (0,.5)--(.5,0);
\draw (0,.5)--(.25,0);
\filldraw[fill=white,draw=black] (-.25,-.5) circle (3pt);
\node at (-.25,-.5) [right]{$\scriptstyle{\phi}$};
\filldraw (0,.5) circle (3pt);
\node at (0,.5) [right]{$\scriptstyle{x}$};
\end{tikzpicture}
+
\begin{tikzpicture}[baseline=-\the\dimexpr\fontdimen22\textfont2\relax] 
\draw (0,0)--(0,-1);
\draw (0,1)--(0,0);
\draw (0,.5)--(-.25,0);
\draw (0,.5)--(-.5,0);
\draw (0,.5)--(.5,0);
\draw (0,.5)--(.25,0);
\filldraw[fill=white,draw=black] (0,-.5) circle (3pt);
\node at (0,-.5) [right]{$\scriptstyle{\phi}$};
\filldraw (0,.5) circle (3pt);
\node at (0,.5) [right]{$\scriptstyle{x}$};
\end{tikzpicture}
+
\begin{tikzpicture}[baseline=-\the\dimexpr\fontdimen22\textfont2\relax] 
\draw (.25,0)--(.25,-1);
\draw (0,1)--(0,0);
\draw (0,.5)--(-.25,0);
\draw (0,.5)--(-.5,0);
\draw (0,.5)--(.5,0);
\draw (0,.5)--(.25,0);
\filldraw[fill=white,draw=black] (.25,-.5) circle (3pt);
\node at (.25,-.5) [right]{$\scriptstyle{\phi}$};
\filldraw (0,.5) circle (3pt);
\node at (0,.5) [right]{$\scriptstyle{x}$};
\end{tikzpicture}
+
\begin{tikzpicture}[baseline=-\the\dimexpr\fontdimen22\textfont2\relax] 
\draw (.5,0)--(.5,-1);
\draw (0,1)--(0,0);
\draw (0,.5)--(-.25,0);
\draw (0,.5)--(-.5,0);
\draw (0,.5)--(.5,0);
\draw (0,.5)--(.25,0);
\filldraw[fill=white,draw=black] (.5,-.5) circle (3pt);
\node at (.5,-.5) [right]{$\scriptstyle{\phi}$};
\filldraw (0,.5) circle (3pt);
\node at (0,.5) [right]{$\scriptstyle{x}$};
\end{tikzpicture}
\bigg)
+{}\\
-(-1)^{k} \bigg(
\begin{tikzpicture}[baseline=-\the\dimexpr\fontdimen22\textfont2\relax] 
\draw (0,1.5)--(0,.5);
\draw (0,1)--(-.5,.5)--(-.5,0)--(-1,-.5);
\draw (-.5,0)--(0,-.5);
\draw (0,1)--(.5,.5)--(.5,-1)--(0,-1.5);
\draw (.5,-1)--(1,-1.5);
\filldraw[fill=white,draw=black] (0,1) circle (3pt);
\node at (0,1) [right]{$\scriptstyle{\ul{x}_1}$};
\filldraw (-.5,0) circle (3pt);
\node at (-.5,0) [right]{$\scriptstyle{x_2}$};
\filldraw (.5,-1) circle (3pt);
\node at (.5,-1) [right]{$\scriptstyle{x_3}$};
\end{tikzpicture}
+(-1)^{(k-1)\dg{x_1}}
\begin{tikzpicture}[baseline=-\the\dimexpr\fontdimen22\textfont2\relax] 
\draw (0,1.5)--(0,.5);
\draw (0,1)--(-.5,.5)--(-.5,0)--(-1,-.5);
\draw (-.5,0)--(0,-.5);
\draw (0,1)--(.5,.5)--(.5,-1)--(0,-1.5);
\draw (.5,-1)--(1,-1.5);
\filldraw (0,1) circle (3pt);
\node at (0,1) [right]{$\scriptstyle{x_1}$};
\filldraw[fill=white,draw=black] (-.5,0) circle (3pt);
\node at (-.5,0) [right]{$\scriptstyle{\ul{x}_2}$};
\filldraw (.5,-1) circle (3pt);
\node at (.5,-1) [right]{$\scriptstyle{x_3}$};
\end{tikzpicture}
+(-1)^{(k-1)(\dg{x_1}+\dg{x_2})}
\begin{tikzpicture}[baseline=-\the\dimexpr\fontdimen22\textfont2\relax] 
\draw (0,1.5)--(0,.5);
\draw (0,1)--(-.5,.5)--(-.5,0)--(-1,-.5);
\draw (-.5,0)--(0,-.5);
\draw (0,1)--(.5,.5)--(.5,-1)--(0,-1.5);
\draw (.5,-1)--(1,-1.5);
\filldraw (0,1) circle (3pt);
\node at (0,1) [right]{$\scriptstyle{x_1}$};
\filldraw (-.5,0) circle (3pt);
\node at (-.5,0) [right]{$\scriptstyle{x_2}$};
\filldraw[fill=white,draw=black] (.5,-1) circle (3pt);
\node at (.5,-1) [right]{$\scriptstyle{\ul{x}_3}$};
\end{tikzpicture}
+\cdots \bigg).
\end{gather*}

\begin{theorem} \label{THMResOfDer}
If $X$ is nonpositively graded\footnote{Recall we use the \emph{co}homological grading.}, then the map $\rho_{\oDR{k}}:\oDR{k}\to\oDP{k}$, defined by
\begin{align*}
&\rho_{\oDR{k}} (x) := \rho_{\oR}(x), \\
&\rho_{\oDR{k}} (\phi) := \phi, \\
&\rho_{\oDR{k}} (\ul{x}) := 0.
\end{align*}
is a quism (a.k.a. quasi-isomorphism), hence a quasi-free resolution of the operad $\oDP{k}$.
\end{theorem}

\begin{proof}
$\dd_{\oDR{k}}^2 = 0$ follows as in Lemma $3.4$ of \cite{GS}, except the sign check is slightly more difficult with degrees of $\phi$ and $\ul{x}$ shifted.

To verify that $\rho_{\oDR{k}}$ is a quism, we slightly modify the proof of Theorem $3.5$ of \cite{GS}.
For reader's convenience, we present full details here.
Additional explanations can be found in loc. cit.

We introduce an additional ``$\gr_1$'' grading
$$\gr_1(x):=-\dg{x},\quad \gr_1(\phi):=1,\quad\gr_1(\ul{x}):=-\dg{\ul{x}}+1$$
on the generators from $\oDR{k}$ and extend it by requiring the operadic composition to be of $\gr_1$ degree $0$.
An exhaustive filtration $0=\mathfrak{F}_{-1}\subset\mathfrak{F}_0\subset\cdots$ of $\oDR{k}$ is defined by
$$\mathfrak{F}_n:=\set{z\in\oDR}{\gr_1(z)\leq n}.$$
$\oDP{k}$ is equipped with a trivial filtration $0\subset\oDP{k}\subset\cdots$.
Consider the spectral sequences $E^*$ and $E'^*$ associated to these filtrations.
We will show that $\rho_{\oDR{k}}$ induces quism $(E^1,\dd^1)\xrightarrow{\sim} (E'^1,\dd'^1)$ and then we use a standard comparison theorem for spectral sequences to conclude the proof.

The differential $\dd^0$ on  $E^0$ is given by
$$\dd^0(x)=0=\dd^0(\phi),\quad \dd^0(\ul{x})=\phi\oo_1 x-(-1)^{k\dg{x}}\sum_{i=1}^n x\oo_i \phi$$
for any $x\in X(n)$.
For the moment, assume
\begin{gather} \label{EQHomologyOfE0}
H_*(E^0,\dd^0) \cong \frac{\Fr{X\op\Phi}}{(\phi\oo_1 x-(-1)^{k\dg{x}}\sum_{i} x\oo_i \phi)}
\end{gather}
as operads.
This is $E^1$ and the corresponding differential is $\dd^1(x)=\dd_{\oR}(x),\ \dd^1(\phi)=0$.
Next, observe that $E'^1$ is $\oDP{k}\cong\oP\oo\Fr{\Phi}$ as graded $\Sigma$ modules, and we transfer the differential and the operadic composition to $\oP\oo\Fr{\Phi}$.
Similarly, $E^1 \cong \Fr{X\op\Phi}/(\phi\oo_1 x\pm\sum_{i} x\oo_i \phi) \cong \Fr{X}\oo\Fr{\Phi}$.
Under these isos, $\rho_{\oDR{k}}^1:E^1\to E'^1$ becomes $\rho_{\oR}\oo\id:\Fr{X}\oo\Fr{\Phi}\to\oP\oo\Fr{\Phi}$, which is a quism by the K\"unneth formula for the $\oo$ product of $\Sigma$ modules (e.g. Proposition $6.2.5$ of \cite{LV}).

It remains to prove \eqref{EQHomologyOfE0}.
Let $g_1,\ldots,g_m\in X\sqcup \Phi\sqcup \ul{X}$ be composed along a tree in $\Fr{X\sqcup \Phi\sqcup \ul{X}}$.
Denote $v_i$ the vertex decorated by $g_i$.
We say that $g_i$ is in depth $d$ in the composition iff the shortest path from $v_i$ to the root vertex passes through exactly $d$ vertices (including $v_i$ and the root vertex) decorated by elements of $X\sqcup\ul{X}$.
Let $\oDR{k}^n$ be the $\Sigma$ module spanned by all tree compositions in $\oDR{k}$ whose every generator from $\ul{X}$ is in depth at most $n$.
Obviously
$$\Fr{X\op\Phi}=\oDR{k}^0\subset\oDR{k}^1\subset\cdots\to\colim_n \oDR{k}^n\cong\oDR{k},$$
where the colimit is taken in the category of $\Sigma$ modules.
For $n\geq 1$, let $\oQ^n$ be the $\Sigma$ module which is a quotient of $\Fr{X\op\Phi}$ by the sub $\Sigma$ module spanned by tree compositions of $g_1,\ldots,g_n$, where each $g_i$ is from $X\sqcup\Phi$ except for at least one $g_j$ which is in depth at most $n$ and is of the form $\phi\oo_1 x-(-1)^{k\dg{x}}\sum_{i}x\oo_i \phi$ for some $x\in X$.
There are obvious projections
$$\Fr{X\op\Phi}=\mathcal{Q}^0 \onto \mathcal{Q}^1 \onto \cdots\to\colim_n \mathcal{Q}^n\cong \frac{\Fr{X\op\Phi}}{(\phi\oo_1 x - (-1)^{k\dg{x}}\sum_{i}x\oo_i \phi)}.$$
Observe that
\begin{gather} \label{EQObsOnQs}
\oQ^{n+1} \cong X\oo\oQ^n
\end{gather}
as $\Sigma$ modules.
For the moment, assume
\begin{gather} \label{EQHomologyOfDRkn}
H_*(\oDR{k}^n) \cong \oQ^n
\end{gather}
as $\Sigma$ modules.
Then
$$H_*(\oDR{k},\dd^0)\cong H_*(\colim_n\oDR{k}^n)\cong\colim_n H_*(\oDR{k}^n) \cong \colim_n \mathcal{Q}^n \cong \frac{\Fr{X\op\Phi}}{(\phi\oo_1 x\pm\sum_{i}x\oo_i \phi)}$$
and the resulting iso $H_*(\oDR{k},\dd^0)\cong\Fr{X\op\Phi}/(\phi\oo_1 x\pm\sum_{i}x\oo_i \phi)$ is easily seen to be an iso of operads.

So it remains to check \eqref{EQHomologyOfDRkn}.
We proceed by induction: assume it holds for $n$ and we want to prove it for $n+1$.
Denote $\phi^l:=\phi\oo_1\cdots\oo_1\phi$, the $l$-fold operadic composition.
For $\ul{x}\in\ul{X}(m)$ and $x_1,\ldots,x_m\in\oDR{k}^n$, we have, in $\oDR{k}^{n+1}$, the following formula:
\begin{gather*}
(-1)^{kl}\dd^0\left(\phi^l\oo_1\ul{x}\oo(x_1,\ldots,x_m)\right) = \\
= \phi^{l+1}\oo_1 x\oo(x_1,\ldots,x_m) -(-1)^{k\dg{x}}\sum_{i=1}^m\phi^m\oo_1 x\oo_i\phi\oo(x_1,\ldots,x_l) +{} \\
+ (-1)^{\dg{x}+k-1}\sum_{i=1}^l(-1)^{\sum_{j=1}^{i-1}\dg{x_j}}\phi^l\oo_1\ul{x}\oo(x_1,\ldots,\dd^0(x_i),\ldots,x_m)
\end{gather*}
We introduce an additional ``$\gr_2$'' grading
$$\gr_2(x):=0=:\gr_2(\phi),\quad\gr_2(\ul{x}):=1$$
on the generators from $\oDR{k}$ and extend it by requiring the operadic composition to be of $\gr_2$ degree $0$.
This induces a grading on $\oDR{k}^n$.
Let $\mathfrak{G}_p\subset\oDR{k}^{n+1}$ be spanned by $\phi^l\oo_1 g\oo(x_1,\ldots,x_m)$ satisfying $g\in(X\sqcup\ul{X})(m)$, $x_i\in\oDR{k}^n$ for $1\leq i\leq m$ and $\sum_{i=1}^m\gr_2(x_i)\leq p$.
Obviously $0=\mathfrak{G}_{-1}\subset\mathfrak{G}_0\subset\cdots$ is an exhaustive filtration of $\oDR{k}^{n+1}$.
$\oQ^{n+1}$ is equipped with a trivial filtration $0\subset\oQ^{n+1}\subset\cdots$.
Consider the spectral sequences $E^{0*}$ and $E'^{0*}$ associated to these filtrations.
We will prove that the projection $\pr:\oDR{k}^{n+1}\onto\oQ^{n+1}$ induces a quism $\pr^1:(E^{01},\dd^{01})\to(E'^{01},\dd'^{01})$ and then we use the comparison theorem to conclude the proof.

The differential $\dd^{00}$ on $E^{00}$ is
\begin{gather*}
(-1)^{kl}\dd^{00}( \phi^l\oo_1\ul{x}\oo(x_1,\ldots,x_m) ) = \\
\phi^{l+1}\oo_1x\oo(x_1,\ldots,x_m) - (-1)^{k\dg{x}}\sum_{i=1}^\ar{x}\phi^l\oo_1 x\oo_i\phi\oo(x_1,\ldots,x_m).
\end{gather*}
Observe that $\Ker\dd^{00}=\Fr{\Phi}\oo X\oo\oDR{k}^n$ and hence $H_*(E^{00},\dd^{00}) \cong X\oo\oDR{k}^n$ as graded $\Sigma$ modules and we transport the differential and operadic composition as usual.
Hence $E^{01}$ is $X\oo\oDR{k}^n$ and $\dd^{01}$ is a restriction of $\dd^{0}$.
Thus we get $H^*(E^{01},\dd^{01})\cong X\oo H^*(\oDR{k}^n,\dd^0)\cong X\oo\oQ^n \cong \oQ^{n+1}$ using the K\"unneth formula, induction hypothesis and \eqref{EQObsOnQs}.
Thus $\pr^1$ is a quism.
\end{proof}

\begin{remark}
The assumption that $X$ is nonpositively graded in Theorem \ref{THMResOfDer} can be modified.
We used it to get $\mathfrak{F}_{-1}=0$ in the above proof.
Boundedness of the filtration $\mathfrak{F}_*$ then guarantees convergence of the spectral sequence $E^*$.

For example, if $X(0)=0=X(1)$ and $X$ consists of unbounded chain complexes, the filtration $\mathfrak{F}_*$ is not bounded below.
However, an easy tree combinatorics shows that for every arity $n\geq 0$ the filtration $\mathfrak{F}_*(n)$ is bounded below.
This is sufficient for the proof to work as before.
\end{remark}

\begin{remark} \label{REMDistrib}
The operad $\oDP{k}$ can be described using a distributive law $\Fr{\Phi}\circ\oP\To{\Lambda}\oP\circ\Fr{\Phi}$ given by
$$\phi \oo_1 p \mapsto (-1)^{k\dg{p}}\sum_{i=1}^n p\oo_i\phi$$
for $p\in\oP(n)$.
We leave it to the reader to generalize the formula for the case of $\phi$ replaced by $\phi\oo\cdots\oo\phi$ on the LHS.
It is easy to verify
$$\oDP{k} = \frac{\oP \fpr \Fr{\Phi}}{\left(\phi\oo_1 p-\Lambda(\phi\oo_1 p)\right)} \cong \oP\circ\Fr{\Phi}$$
as graded $\Sigma$ modules.

In case the operad $\oP$ is Koszul, the methods of Section $8.6$ of \cite{LV} immediately imply that $\oDP{k}$ is also Koszul.
This gives an alternative way to construct a resolution of $\oDP{k}$ in this special case, which we will discuss in Appendix.
Also, this has already been done in the non-$\Sigma$ setting and with homological grading for $\oP=\oAss$ in \cite{LD}.
The resulting resolution coincides with that of Theorem \ref{THMResOfDer} only up to a sign - there is a mistake in the formula for $\dd$ above Proposition $7.6$ in \cite{LD} as one checks that $\dd^2 Dm_3\neq 0$ in the notation of loc. cit.
\end{remark}

\begin{proposition} \label{PROPCofibrantDkR}
If $\oR$ is a quasi-free cofibrant dg operad, then so is $\oDR{k}$, for any $k\in\Z$.
\end{proposition}

\begin{proof}
Recall from \cite{MHAAHA} that for the quasi free dg operad $\oR=(\Fr{X},\partial_{\oR})$, cofibrancy means that there is a grading $X=\bigoplus_{i\geq 0}X_{[i]}$ such that $\partial_{\oR}X_{[n]}\subset\Fr{\bigoplus_{i<n}X_{[i]}}$.
Extend this grading to $X\op\Phi\op\ul{X}$, the space of generators of $\oDR{k}$, by requiring $\phi$ to have degree $0$ and $\ul{X}_{[n]}:=\upar^{k-1}(X_{[n-1]})$, where the suspension acts on the cohomological grading.
With this extended grading we have cofibrancy of $\oDR{k}$.
\end{proof}

It is reasonable to define a $\oP'_\infty$ algebra on a dg vector space $A$ as an $\oR'$ algebra on $A$, where $\oR'$ is a quasi-free cofibrant resolution of $\oP'$.
The cofibrancy guarantees nice homotopy properties \cite{MHAAHA}, \cite{BM}.
The quasi-freeness makes computation easier.
Importantly, quasi-free cofibrant resolutions always exist, e.g. the bar-cobar resolution of Theorem $6.6.5$ of \cite{LV}.
Similar definitions of strong homotopy algebra appear in e.g. \cite{MHAVRO} and \cite{LV}.

Stated this way, the notion of $\oP'_\infty$ algebra is ambiguous, since the quasi-free resolution is not unique.
Thus when invoking a $\oP'_\infty$ algebra, we always assume a particular choice of the quasi-free resolution $\oR'$ has been made.
This is the case in practice: e.g. $A_\infty$ algebras refer to the Koszul resolution of the operad $\oAss$.

As an aside, recall that this ambiguity can be essentially removed (in many cases) by requiring the resolution to be minimal, since then it is unique up to an isomorphism by Proposition $3.7$ of \cite{DCV}.
The Koszul resolutions are examples of minimal models.

Now we apply the above for $\oP':=\oDP{k}$ and $\oR':=\oDR{k}$.
Denote by $\N_0$ the set of all natural numbers and $0$.

\begin{definition}
Let $\beta:\oDR{k}\to\oEnd{A}$ be a $\oDR{k}$ algebra.
There is an injection $\oR=\Fr{X}\into\Fr{X}\fpr\Fr{\Phi\op\ul{X}}\cong\Fr{X\op\Phi\op\ul{X}}=\oDR{k}$.
Using it, $\beta|_{\oR}$ can be seen as a $\oP_\infty$ algebra, which is determined by the collection
$$\set{\beta(x):A^{\ot n}\to A}{n\in\N_0,\ x\in X'(n)}$$
of operations, where $X'(n)$ is any subset of $X(n)$ generating $X(n)$ as a $\fld\Sigma_n$-module.
Then the collection
$$\{\beta(\phi)\}\cup\set{\beta(\ul{x}):A^{\ot n}\to A}{n\in\N_0,\ x\in X'(n)}$$
is called a \emph{degree $k$ strong homotopy derivation} of the $\oP_\infty$ algebra.
\end{definition}

Recall that $\Fr{X}$ carries the \emph{weight grading}
$$\Fr{X} = \bigoplus_{n\geq 0}\Fr[n]{X},$$
where $\Fr[n]{X}$ consists of (linear combinations of) elements corresponding to trees with exactly $n$ vertices.

\begin{proposition} \label{PROSelfExample}
If the differential $\dd_{\oR}$ of the quasi-free resolution $\oR$ of $\oP$ is quadratic, i.e. $\dd_{\oR}(X)\subset\Fr[2]{X}$, then any $\oP_\infty$ algebra $\beta:\oR\to\oEnd{A}$ has a degree $1$ strong homotopy derivation $\beta:\oDR{1}\to\oEnd{A}$ defined by
$$\beta(\phi):=-d, \quad \beta(\ul{x}):=\beta(x),\ x\in X.$$
In other words, the differential $d$ on $A$ and the structure operations of the $\oP_\infty$ algebra constitute a strong homotopy derivation of the $\oP_\infty$ algebra.
\end{proposition}


\begin{proof}
It suffices to verify that $\beta$ commutes with differentials.
\begin{gather*}
\beta\dd_{\oDR{k}}(\ul{x}) = -d \oo_1 \beta(x) +(-1)^{\dg{x}} \sum_{i=0}^n \beta(x)\oo_i d \ + \beta s(\dd_\oR x).
\end{gather*}
The quadraticity assumption implies that
$$\dd_{\oR}(x) = \sum_j \sigma_j (x'_j\oo_{i_j}x''_j)$$
for some $x'_j, x''_j\in X$ and $i_j\in\N$ and some permutations $\sigma_j$.
Thus
\begin{align*}
\beta s\dd_\oR(x) &= \beta\bigg( \sum_j \sigma_j (\ul{x'_j}\oo_{i_j}x''_j) + \sum_j \sigma_j (x'_j\oo_{i_j}\ul{x''_j}) \bigg) = 2\beta\dd_\oR(x) = 2\dd_{\oEnd{A}}\beta(x) = \\
&= 2( d \oo_1 \beta(x) -(-1)^{\dg{\beta(x)}} \sum_{i=0}^n \beta(x)\oo_i d )
\end{align*}
since $\dg{s}=0$.
$\beta\dd_{\oDR{k}}(\ul{x})=\dd_{\oEnd{A}}\beta(\ul{x})$ follows.
\end{proof}

\begin{example} \label{EXTautologicalHomDer}
The quadraticity assumption of Proposition \ref{PROSelfExample} on $\oR$ is satisfied by the Koszul resolution of any Koszul operad $\oP$.
In particular, we get an example of degree $1$ strong homotopy derivation of an arbitrary $A_\infty$ or $L_\infty$ algebra.
\end{example}


\begin{example}
Let $\oR\to\oP$ be a quadratic quasi-free resolution which is a fibration, i.e. aritywise surjection.
Let $\oR'\to\oP$ be an arbitrary quasi-free cofibrant resolution.
Then Proposition \ref{PROPCofibrantDkR} and the lifting property of cofibrations imply that there exists an operad morphism $\oDR{k}'\to\oDR{k}$.
Thus the degree $1$ strong homotopy derivation of an $\oR$ algebra on $A$ constructed in Proposition \ref{PROSelfExample} induces a degree $1$ strong homotopy derivation of an $\oR'$ algebra on $A$.

In particular, for $\oP$ Koszul, taking $\oR$ to be its Koszul resolution, this gives an example of degree $1$ strong homotopy derivation of an algebra over the bar-cobar construction of $\oP$.
\end{example}

\begin{remark}
If $\oP$ of Proposition \ref{PROSelfExample} is concentrated in degree $0$, then it is already Koszul by Theorem $39$ of \cite{MerkVall}.
\end{remark}

\begin{remark} \label{REMKoszulDiscuss}
Obviously, $\oDR{k}$ is minimal iff $\oR$ is minimal.
Let $\oP$ be a Koszul operad and let $\oR$ be its Koszul resolution.
Since Koszul resolutions are always minimal, $\oDR{k}$ is minimal.
In Remark \ref{REMDistrib}, we have observed that in this case $\oDP{k}$ is Koszul.
The Koszul resolution of $\oDP{k}$ is minimal and, by the uniqueness of minimal resolutions, it must be isomorphic to $\oDR{k}$.
\end{remark}


\section{\texorpdfstring{Strong homotopy derivations of $A_\infty$ algebras}{Strong homotopy derivations of A-infinity algebras}} \label{SECAoo}


\subsection{Application of the theory}

We make degree $k$ strong homotopy derivations of strong homotopy associative algebras explicit.
The calculation is mostly taken from \cite{GS}:

Let $\oP$ be the operad for associative algebras, that is 
$$\oP:=\oAss=\Fr{\fld\Sigma_2\{\mu\}}/(\mu\oo_1\mu-\mu\oo_2\mu).$$
Its minimal resolution is (see e.g. \cite{MHAVRO} with a $(-1)^{n+1}$ sign shift) $\oR:=(\Fr{X},\dd_\oR)\To{\rho_\oR}(\oAss,0)$, where $X$
is the $\Sigma$ module whose arity $n$ component is the $n!$-dimensional space $\fld\Sigma_n\{x^n\}$, where $\dg{x^n}=2-n$ and $\dd_\oR$ is a derivation differential given by
$$\dd_\oR(x^n):=\sum_{i+j=n+1}\sum_{l=1}^i (-1)^{i+(l+1)(j+1)}x^i\oo_l x^j$$
and the quism $\rho_\oR:\oR\to\oP$ is given by
$$\rho_\oR(x^2):=\mu, \quad \rho_\oR(x^n):=0\mbox{ for }n\geq 3.$$
Recall that $\oR$ algebras on a dg vector space are $A_\infty$ algebras.
The associated \emph{operad with derivation} is
$$\oDP{k}:=\frac{\oAss\fpr \Fr{\Phi}}{(\phi\oo\mu-\mu\oo_1\phi-\mu\oo_2\phi)},$$
where $\Phi:=\fld\{\phi\}$ with $\phi$ a degree $k$ element of arity $1$.
Its quasi-free cofibrant resolution is $$\oDR:=(\Fr{X\op\Phi\op\ul{X}},\dd_{\oDR{k}})\To{\rho_{\oDR{k}}}(\oDP{k},0),$$
where the differential $\dd_{\oDR{k}}$ is given by
\begin{align*}
&\dd_{\oDR{k}}(x^n) := \sum_{i+j=n+1}\sum_{l=1}^i (-1)^{i+(l+1)(j+1)}x^i\oo_l x^j,\\
&\dd_{\oDR{k}}(\phi) := 0, \\
&\dd_{\oDR{k}}(\ul{x}^n) := \phi\oo_1 x^n - (-1)^{nk}\sum_{l=1}^n x^n\oo_l\phi + {} \\
&\phantom{\dd_{\oDR{k}}(\ul{x}^n) := } - \sum_{i+j=n+1}\sum_{l=1}^i (-1)^{k+i+(l+1)(j+1)} (\ul{x}^i\oo_l x^j + (-1)^{(k+1)i} x^i\oo_l\ul{x}^j)
\end{align*}
and the quism $\rho_{\oDR{k}}$ by
$$\rho_{\oDR{k}}(x^n):=\rho_\oR(x^n),\quad \rho_{\oDR{k}}(\phi):=\phi,\quad \rho_{\oDR{k}}(\ul{x}^n)=0.$$

Thus a $\oDR{k}$ algebra $\beta:\oDR{k}\to\oEnd{A}$ on a dg vector space $(A,d)$ consists of two collections $\set{m_n}{n\geq 2}$ and $\set{\theta_n}{n\geq 1}$ of operations
\begin{gather*}
m_n:A^{\ot n}\to A,\quad \dg{m_n}=2-n, \\
\theta_n:A^{\ot n}\to A,\quad \dg{\theta_n}=k-n+1
\end{gather*}
given by $m_n:=\beta(x^n)$, $\theta_1:=(-1)^k\beta(\phi)$ and $\theta_n:=\beta(\ul{x}^n)$ for $n\geq 2$,
satisfying
\begin{flalign}
&d \oo_1 m_n-(-1)^{n}\sum_{l=1}^n m_n\oo_l d = \sum_{i+j=n+1}\sum_{l=1}^i (-1)^{i+(l+1)(j+1)}m_i\oo_l m_j, \quad n\geq 2, \label{EQAinfty} \\
&\begin{multlined}
d\oo_1\theta_n-(-1)^{k-n+1}\sum_{l=1}^n \theta_n\oo_l d = \\
= \sum_{i+j=n+1}\sum_{l=1}^i (-1)^{k+1+i+(l+1)(j+1)} (\theta_i\oo_l m_j + (-1)^{(k+1)i} m_i\oo_l\theta_j), \quad n\geq 1. \nonumber
\end{multlined}
\end{flalign}
The indices $i,j,l$ are restricted to values where both sides have sense.

The above properties of $m_n$'s and $\theta_n$'s are expressed by ``$\{\theta_n\}$ is a strong homotopy derivation of the $A_\infty$ algebra $(A,\{m_n\})$''.
Under the notation
$$
\begin{tikzpicture}[baseline=-\the\dimexpr\fontdimen22\textfont2\relax] 
\draw (0,.5)--(0,-.5);
\draw (0,0)--(-.5,-.5);
\draw (0,0)--(-.25,-.5);
\draw (0,0)--(.25,-.5);
\draw (0,0)--(.5,-.5);
\filldraw (0,0) circle (3pt);
\node at (0,0) [right]{$\scriptstyle{n}$};
\end{tikzpicture} 
:= m_n, \qquad
\begin{tikzpicture}[baseline=-\the\dimexpr\fontdimen22\textfont2\relax] 
\draw (0,.5)--(0,-.5);
\draw (0,0)--(-.5,-.5);
\draw (0,0)--(-.25,-.5);
\draw (0,0)--(.25,-.5);
\draw (0,0)--(.5,-.5);
\filldraw[fill=white,draw=black] (0,0) circle (3pt);
\node at (0,0) [right]{$\scriptstyle{n}$};
\end{tikzpicture}
:= \theta_n,
$$
the usual mnemonic for such formulas is
\begin{gather*}
\dd_{\oEnd{A}} 
\begin{tikzpicture}[baseline=-\the\dimexpr\fontdimen22\textfont2\relax] 
\draw (0,.5)--(0,-.5);
\draw (0,0)--(-.5,-.5);
\draw (0,0)--(-.25,-.5);
\draw (0,0)--(.25,-.5);
\draw (0,0)--(.5,-.5);
\filldraw (0,0) circle (3pt);
\node at (0,0) [right]{$\scriptstyle{n}$};
\end{tikzpicture} 
= 
\sum_{i+j=n+1}\sum_{l=1}^i (-1)^{i+(l+1)(j+1)} 
\begin{tikzpicture}[baseline=-\the\dimexpr\fontdimen22\textfont2\relax] 
\draw (0,1)--(0,-1);
\draw (0,.5)--(-.5,0);
\draw (0,.5)--(-.25,0);
\draw (0,.5)--(.25,0);
\draw (0,.5)--(.5,0);
\draw (0,-.5)--(-.5,-1);
\draw (0,-.5)--(-.25,-1);
\draw (0,-.5)--(.25,-1);
\draw (0,-.5)--(.5,-1);
\filldraw (0,.5) circle (3pt);
\node at (0,.5) [right]{$\scriptstyle{i}$};
\filldraw (0,-.5) circle (3pt);
\node at (0,-.5) [right]{$\scriptstyle{j}$};
\node at (0,0) [below left]{$\scriptstyle{l}$};
\end{tikzpicture} \ \ , \\
\dd_{\oEnd{A}} 
\begin{tikzpicture}[baseline=-\the\dimexpr\fontdimen22\textfont2\relax] 
\draw (0,.5)--(0,-.5);
\draw (0,0)--(-.5,-.5);
\draw (0,0)--(-.25,-.5);
\draw (0,0)--(.25,-.5);
\draw (0,0)--(.5,-.5);
\filldraw[fill=white,draw=black] (0,0) circle (3pt);
\node at (0,0) [right]{$\scriptstyle{n}$};
\end{tikzpicture}
=
\sum_{i+j=n+1}\sum_{l=1}^i (-1)^{k+1+i+(l+1)(j+1)} \left(
\begin{tikzpicture}[baseline=-\the\dimexpr\fontdimen22\textfont2\relax] 
\draw (0,1)--(0,-1);
\draw (0,.5)--(-.5,0);
\draw (0,.5)--(-.25,0);
\draw (0,.5)--(.25,0);
\draw (0,.5)--(.5,0);
\draw (0,-.5)--(-.5,-1);
\draw (0,-.5)--(-.25,-1);
\draw (0,-.5)--(.25,-1);
\draw (0,-.5)--(.5,-1);
\filldraw[fill=white,draw=black] (0,.5) circle (3pt);
\node at (0,.5) [right]{$\scriptstyle{i}$};
\filldraw (0,-.5) circle (3pt);
\node at (0,-.5) [right]{$\scriptstyle{j}$};
\node at (0,0) [below left]{$\scriptstyle{l}$};
\end{tikzpicture}
 + (-1)^{(k+1)i} 
\begin{tikzpicture}[baseline=-\the\dimexpr\fontdimen22\textfont2\relax] 
\draw (0,1)--(0,-1);
\draw (0,.5)--(-.5,0);
\draw (0,.5)--(-.25,0);
\draw (0,.5)--(.25,0);
\draw (0,.5)--(.5,0);
\draw (0,-.5)--(-.5,-1);
\draw (0,-.5)--(-.25,-1);
\draw (0,-.5)--(.25,-1);
\draw (0,-.5)--(.5,-1);
\filldraw (0,.5) circle (3pt);
\node at (0,.5) [right]{$\scriptstyle{i}$};
\filldraw[fill=white,draw=black] (0,-.5) circle (3pt);
\node at (0,-.5) [right]{$\scriptstyle{j}$};
\node at (0,0) [below left]{$\scriptstyle{l}$};
\end{tikzpicture}
\right).
\end{gather*}


\subsection{Suspension} \label{SECSusp}

The above formulas can be simplified by defining $V:=\downarrow\!\! A$.
A map $f:A^{\ot n}\to A$ of degree $d$ is equivalent to a map $f':V^{\ot n}\to V$ of degree $d+n-1$ by the formula
$$f'= \ \downarrow\!\!f\uparrow^{\ot n}.$$
The passage to the suspension $V$ alters signs; for example, we have the following useful formula:
\begin{gather} \label{EQSuspension}
(f_i\oo_l g_j)' = \downar (f_i\oo_l g_j)\upar^{\ot i+j-1} = (-1)^{|g_j|(i+1)+ij+jl+i+l}f'_i\oo_l g'_j,
\end{gather}
where $f_i:A^{\ot i}\to A$ and $g_j:A^{\ot j}\to A$.
In the sequel, we will omit the prime and thus we write $f$ for both $f$ and $f'$.

\begin{proposition} \label{shder}
An $A_\infty$ algebra  structure is equivalently given by a collection of degree one linear maps
$m_n:V^{\otimes n}\rightarrow V, n\geq 1$ that satisfy the relations
$$\sum_{k+l=n+1}\sum_{i=1}^k(-1)^{\alpha} m_k(v_1,\dots,v_{i-1},m_l(v_i,\dots,v_{i+l-1}),v_{i+l},\dots,v_n)=0$$
for $n\geq 1$ and $\alpha$ is the sum of the degrees of the elements $v_1,\dots,v_{i-1}$.

A strong homotopy derivation of degree $k$ of an $A_\infty$ algebra $(V,\{m_n\})$ is equivalently given by a collection of degree $k$ linear maps $\theta_q:V^{\otimes q}\rightarrow V,q\geq 1,$ that satisfy the relations
\begin{gather*}
0=\sum_{r+s=q+1}\sum_{i=0}^{r-1}(-1)^{\beta}\theta_r(v_1,\dots, v_i,m_s(v_{i+1},\dots,v_{i+s}),\dots,v_q) + {} \label{EQAooHomDer} \\
-(-1)^k(-1)^{\gamma}m_r(v_1,\dots,v_i,\theta_s(v_{i+1},\dots,v_{i+s}),\dots,v_q)
\end{gather*}
for $q\geq 1$.
The exponent $\beta$ results from moving the degree one maps $m_s$  past $(v_1,\dots,v_i)$ and  $\gamma$ from moving the degree $k$ maps $\theta_s$ past  $(v_1,\dots,v_i)$.
\end{proposition}

The above description of $A_\infty$ algebras is often used as a definition, e.g. \cite{KS} (set $m_0=0$ in their definition).
It is equivalent to the usual definition, e.g. \cite{S}, where the operations have degree $2-n$.
Also, the same notion of homotopy derivation appeared, for $k=1$, in \cite{KS}.
The unsuspended version appeared, for $k=0$, in \cite{LD}.

\begin{proof}
We first define $m_1:=d$.
Then we precompose \eqref{EQAinfty} with $\uparrow^{\ot n}$, compose with $\downarrow$ and use \eqref{EQSuspension}:
\begin{gather*}
m_1\oo_1 m_n+\sum_{l=1}^n m_n\oo_l m_1 = -\sum_{i+j=n+1}\sum_{l=1}^i m_i\oo_l m_j, \quad n\geq 2, \\
m_1\oo_1\theta_n-(-1)^{k}\sum_{l=1}^n \theta_n\oo_l m_1 = \sum_{i+j=n+1}\sum_{l=1}^i ((-1)^k\theta_i\oo_l m_j - m_i\oo_l\theta_j), \quad n\geq 1.
\end{gather*}
An evaluation on $v_1,\ldots,v_n\in V$ yields the required relations.
\end{proof}

It is well known \cite{S} that the structure maps $m_n$'s may be extended to a degree $+1$ coderivation $\boldsymbol{m}$ on the tensor coalgebra $T^c(V)$ of $V$, and that the relations are equivalent to the equation $\boldsymbol{m}^2=0$.
Similarly:

\begin{proposition} \label{PROPAinftyHomDerAsCoder}
A degree $k$ strong homotopy derivation $\{\theta_q\}$ of an $A_\infty$ algebra $(V,\{m_n\})$ is equivalently described by a degree $k$ coderivation $\boldsymbol{\theta}:T^c(V)\to T^c(V)$ satisfying
$$[\boldsymbol{m},\boldsymbol{\theta}]=0.$$
\end{proposition}

\begin{proof}
$\boldsymbol{\theta}$ is defined by the projections $\pi_1\boldsymbol{\theta}|_{V^{\otimes q}}=\theta_q$, where $\pi_1:T^c(V)\to V$ is the obvious projection.
$0=[\boldsymbol{m},\boldsymbol{\theta}]=\boldsymbol{m}\circ \boldsymbol{\theta}-(-1)^k\boldsymbol{\theta}\circ\boldsymbol{m}$ is then easily seen to be equivalent to \eqref{EQAooHomDer}.
\end{proof}

\begin{example} \label{EXADoubleBracket}
Let $(V,m)$ be an $A_\infty$ algebra, i.e. $[\bm,\bm]=0$.
Let $\btheta$ be an arbitrary coderivation on $T^c(V)$.
The Jacobi identity then implies that 
\begin{gather*}
[\bm,[\bm,\btheta]] = -(-1)^{1+\dg{\btheta}} [\bm,[\btheta,\bm]] - [\btheta,[\bm,\bm]] = - [\bm,[\bm,\btheta]].
\end{gather*}
Thus $[\bm,[\bm,\btheta]] = 0$.
In other words, $[\bm,\btheta]$ determines a strong homotopy derivation of $(V,m)$.
This gives a reservoir of examples.
\end{example}


\subsection{Strong homotopy inner derivation} \label{SECInner}

As an example of a strong homotopy derivation, we can define a strong homotopy inner derivation of an $A_\infty$ algebra.

First observe that in the case $\oP=\oAss$, Definition \ref{DEFDerivation} of degree $k$ derivation $\theta$ of an $\oP$ algebra $(A,m)$ boils down to the relations $d(\theta(x))=(-1)^k\theta(d(x))$ and $\theta(xy)=\theta(x)y+(-1)^{k|x|}x\theta(y)$.
After passing to $V$, the former relation stays and the latter becomes
$$\theta (xy)=(-1)^k(\theta(x)y)+(-1)^k(-1)^{k|x|}(x\theta(y)).$$
Notice that on $V$, the multiplication has degree $+1$ and satisfies
$$(xy)z+(-1)^{|x|}x(yz)=0.$$
Such a $(V,m)$ is sometimes called a dg anti-associative degree $+1$ algebra \cite{MR}.

\begin{proposition}
Let $a\in V$ have degree $k$ and satisfy $d(a)=0$.
Then the map
$$\theta(x)=ax+(-1)^{k|x|}xa$$
is a derivation of $V$ of degree $k+1$.
We call such a derivation \emph{inner}.
\end{proposition}

\begin{proof}
First we show that
$$d(\theta(x))=(-1)^k\theta(d(x)).$$
The LHS is 
\begin{gather*}
d(ax+(-1)^{k|x|}xa) = d(a)x+(-1)^{k}ad(x)+(-1)^{k|x|}d(x)a+(-1)^{k|x|}(-1)^{\dg{x}}xd(a) = \\
= (-1)^{k}ad(x)+(-1)^{k|x|}d(x)a,
\end{gather*}
while the RHS is
$$(-1)^{k}ad(x)+(-1)^{k}(-1)^{k(1+\dg{x})}d(x)a.$$

Next, we show that
\begin{equation}\label{inner}
\theta(xy)=(-1)^{\alpha}\theta(x)y+(-1)^{\alpha}(-1)^{\alpha|x|}x\theta(y)
\end{equation}
where $\alpha=k+1$.

We have that
\begin{equation}\label{lhs}
\theta(xy)=a(xy)+(-1)^{k(|x|+|y|+1)}(xy)a
\end{equation}
from the definition of $\theta$. 
The terms on the right hand side of \eqref{inner} yield
$$(-1)^{\alpha}\theta(x)y=(-1)^{\alpha} (ax)y+(-1)^{\alpha}(-1)^{k|x|}(xa)y$$
and
$$(-1)^{\alpha}(-1)^{\alpha |x|}x\theta(y)=(-1)^{\alpha}(-1)^{\alpha |x|}x(ay)+(-1)^{\alpha}(-1)^{\alpha |x|}(-1)^{k|y|}x(ya).$$
We see that after re-associating,
$$(-1)^{\alpha}(ax)y=-(-1)^{\alpha}(-1)^ka(xy)=(ax)y,$$
the first term of \eqref{lhs}, and
$$ (-1)^{\alpha)}(-1)^{\alpha |x|}(-1)^{k|y|}x(ya)=-(-1)^{\alpha}(-1)^{\alpha |x|}(-1)^{k|y|}(-1)^{|x|}(xy)a,$$
the second term of \eqref{lhs}.
For the two remaining terms, we have
$$(-1)^{\alpha}(-1)^{k|x|}(xa)y=-(-1)^{\alpha}(-1)^{k|x|}(-1)^{|x|}x(ay)$$
which cancels out the remaining term
$$(-1)^{\alpha}(-1)^{\alpha |x|}x(ay).$$
\end{proof}

Finally, the following generalization of  
inner derivation is now natural:

\begin{proposition} \label{PROInner}
Let $(V,\{m_n\})$ be an $A_\infty$ algebra and let $a\in V$ have the property that the degree of $a$ is $k$ and $m_1(a)=0$.
 Then the maps
\begin{gather} \label{EQInner}
\theta_n(v_1,\dots,v_n)=\sum_{p=0}^n (-1)^{k\sum_{j=1}^p|v_j|}m_{n+1}(v_1,\dots,v_p,a,v_{p+1},\dots,v_n)
\end{gather}
yield a strong homotopy derivation of degree $k+1$.
\end{proposition}

\begin{proof}
To see that
\begin{gather*}
0=\sum_{r+s=q+1}\sum_{i=0}^{r-1}(-1)^{\beta}\theta_r(v_1,\dots, v_i,m_s(v_{i+1},\dots,v_{i+s}),\dots,v_q) + {} \\
-(-1)^{k+1}(-1)^{\gamma}m_r(v_1,\dots,v_i,\theta_s(v_{i+1},\dots,v_{i+s}),\dots,v_q),
\end{gather*}
we first replace the  $\theta_r$ by $m_{r+1}$ and $\theta_s$ by $m_{s+1}$ and $-(-1)^{k+1}$ by $+(-1)^k$
to obtain
\begin{gather*}
\sum_{r+s=q+1}\sum_{i=0}^{r-1}(-1)^{\beta}\{\sum_{p=0}^i(-1)^{k\sum_{j=1}^p|v_j|}m_{r+1}(v_1,\dots,v_p,a,\dots,v_i,m_s(v_{i+1},\dots,v_{i+s}),\dots,v_q) + {} \\
+\sum_{p=i+s}^q(-1)^{k(1+\sum_{j=1}^p|v_j|)}m_{r+1}(v_1,\dots,v_i,m_s(v_{i+1},\dots,v_{i+s}),\dots,v_p,a,\dots,v_q)\} + {} \\
+(-1)^k(-1)^{\gamma}m_r(v_1,\dots,v_i,\sum_{p=i}^{i+s+1}(-1)^{k\sum_{j=i+1}^p|v_j|}m_{s+1}(v_{i+1},\dots,v_p,a,\dots,v_{i+s+1}),\dots,v_q)
\end{gather*}
where $\beta=\sum_{j=1}^i|v_j|$ and $\gamma=(k+1)\sum_{j=1}^i|v_j|$.

We claim that for each fixed position $t$ of the element $a$, we obtain the $A_\infty$ algebra relation which is equal to $0$.  To see this, fix
$1\leq t\leq q+1$ and let $x_i=v_i$ for $i\leq t-1$, $x_t=a$ and $x_{i+1}=v_i$ for $i>t.$
For each $t$ we have
\begin{gather*}
\sum_{r+s=q+1}\{\sum_{i=t}^{q+1-s}(-1)^{\beta}(-1)^{k\sum_{j=1}^{t-1}|v_j|}m_{r+1}(x_1,\dots,x_{t-1},a,\dots,x_i,m_s(x_{i+1},\dots,x_{i+s}),\dots,x_{q+1}) + {} \\
+\sum_{i=0}^{t-s-1}(-1)^{\beta}(-1)^{k(1+\sum_{j=1}^{t-1}|x_j|)}m_{r+1}(x_1,\dots,x_i,m_s(x_{i+1},\dots,x_{i+s}),\dots,x_{t-1},a,\dots,x_{q+1}) + {} \\
+\sum_{i=t-s}^{t-1}(-1)^k(-1)^{\gamma}(-1)^{k\sum_{j=i+1}^{t-1}|v_j|}m_r(x_1,\dots,x_i,m_{s+1}(x_{i+1},\dots,x_{t-1},a,\dots,x_{i+s+1}),\dots,x_{q+1})\}
\end{gather*}
which, after multiplying the first line by $1$ written as $(-1)^k(-1)^k$ and adding in  $0$ written as (recall that $m_1(a)=0$ )
$$(-1)^k(-1)^{k\sum_{j=1}^{t-1}|x_j|}(-1)^{\sum_{j=1}^{t-1}|x_j|}m_{q+1}(x_1,\dots,x_{t-1},m_1(a),x_{t+1},\dots,x_{q+1}),$$
can be seen to equal $(-1)^k(-1)^{k\sum_{j=1}^{t-1}|x_j|}$ times the $A_\infty$ algebra relation for $q+1$ inputs $(v_1,\dots,v_{t-1},a,v_{t+1},\dots,v_q)$.
\end{proof}

As an extension of Example \ref{EXADoubleBracket} we get another proof of Proposition \ref{PROInner} on the coalgebra level.

\begin{proof}[Proof of Proposition \ref{PROInner}]
First, we need to replace $T^c(V)$ by its \emph{unital} version,
$$T^c_u(V) := \fld\op V\op V^{\ot 2}\op\cdots.$$
Recall that any coderivation $\btheta$ on $T^c_u(V)$ is still uniquely determined by its projections $\theta_n:V^{\ot n}\subset T^c_u(V)\To{\btheta}T^c_u(V)\to V$, but this time it includes the $0$-th projection $\theta_0:\fld\to V$.
The formula expressing $\btheta$ in terms of $\theta_n$'s is altered only by allowing the $0$-th projections.
Any coderivation on $T^c(V)$ induces a coderivation on $T^c_u(V)$ by setting its $0$-th projection to $0$.

Now let $a\in V$ and define a coderivation $\theta$ by
$$\theta_0(1):=a, \qquad \theta_n:=0,\quad n\geq 1.$$
Then $[\bm,\btheta]$ is a coderivation\footnote{Again, the formula expressing composition of coderivations in terms of their projections is altered, compared to the non-unital case, only by allowing the $0$-th projections.} on $T^c_u(V)$.
As in Example \ref{EXADoubleBracket}, $[\bm,[\bm,\theta]]=0$.
We wish $[\bm,\btheta]$ restricted to $T^c(V)$ (consequently, it will determine a strong homotopy derivation of $(V,m)$).
This is equivalent to
$$0=[\bm,\btheta]_0=m_1\theta_0=m_1(a).$$
We recognize the assumptions of Proposition \ref{PROInner}.
The projections of the coderivation $[\bm,\btheta]$ are
$$[\bm,\btheta]_n = \sum_{i=1}^n m_{n+1}\oo_i\theta_0$$
and this is exactly \eqref{EQInner}.
\end{proof}



\section{\texorpdfstring{Strong homotopy derivations of $L_\infty$ algebras}{Strong homotopy derivations of L-infinity algebras}} \label{SECLoo}

We make degree $k$ strong homotopy derivations of $L_\infty$ algebras explicit.

Let $\oP$ be the operad for Lie algebras, that is 
$$\oP:=\oLie=\Fr{\fld\{\nu\}}/((1+\kappa+\kappa^2)\cdot(\nu\oo_1\nu)),$$
where $\fld\{\nu\}$ is the $1$-dimensional sign representation of $\Sigma_2$ and $\kappa\in\Sigma_3$ is a cycle of length $3$.
Its minimal resolution is (see e.g. \cite{MHAVRO}) $\oR:=(\Fr{X},\dd_\oR)\To{\rho_\oR}(\oLie,0)$, where $X$ is the $\Sigma$ module whose arity $n$ component is the $1$-dimensional sign representation $\fld\{x^n\}$ of $\Sigma_n$, where $\dg{x^n}=2-n$, and $\dd_\oR$ is a derivation differential given by
$$\dd_\oR(x^n):=\sum_{i+j=n+1} (-1)^{j(i-1)} \sum_{\sigma\in\USh{j,i-1}} \sgn{\sigma} \sigma\cdot(x^i\oo_1 x^j),$$
where $\USh{a,b}=\set{\sigma\in\Sigma_{a+b}}{\sigma(1)<\cdots<\sigma(a),\ \sigma(a+1)<\cdots<\sigma(a+b)}$ is the set of $(a,b)$-unshuffles.
The quism $\rho_\oR:\oR\to\oP$ is given by
$$\rho_\oR(x^2):=\nu, \quad \rho_\oR(x^n):=0\mbox{ for }n\geq 3.$$
Then the associated operad with derivation is
$$\oDP{k}:=\frac{\oLie\fpr \Fr{\Phi}}{(\phi\oo\nu-\nu\oo_1\phi-\nu\oo_2\phi)},$$
where $\Phi:=\fld\{\phi\}$ with $\phi$ a degree $k$ element of arity $1$.
Its quasi-free cofibrant resolution is $$\oDR:=(\Fr{X\op\Phi\op\ul{X}},\dd_{\oDR{k}})\To{\rho_{\oDR{k}}}(\oDP{k},0),$$
where the differential $\dd_{\oDR{k}}$ is given by
\begin{align*}
&\dd_{\oDR{k}}(x^n) := \sum_{i+j=n+1} (-1)^{j(i-1)} \sum_{\sigma\in\USh{j,i-1}} \sgn{\sigma} \sigma\cdot(x^i\oo_1 x^j),\\
&\dd_{\oDR{k}}(\phi) := 0, \\
&\dd_{\oDR{k}}(\ul{x}^n) := \phi\oo_1 x^n - (-1)^{nk}\sum_{l=1}^n x^n\oo_l\phi + {} \\
&\phantom{\dd_{\oDR{k}}(\ul{x}^n) := }- \sum_{i+j=n+1} (-1)^{k+j(i-1)} \hspace{-1ex} \sum_{\sigma\in\USh{j,i-1}} \hspace{-1ex} \sgn{\sigma} \sigma\cdot\left( \ul{x}^i\oo_1 x^j + (-1)^{(k+1)i}x^i\oo_1\ul{x}^j\right)
\end{align*}
and the quism $\rho_{\oDR{k}}$ by
$$\rho_{\oDR{k}}(x^n):=\rho_\oR(x^n),\quad \rho_{\oDR{k}}(\phi):=\phi,\quad \rho_{\oDR{k}}(\ul{x}^n)=0.$$

Thus a $\oDR{k}$ algebra $\beta:\oDR{k}\to\oEnd{A}$ on a dg vector space $(A,d)$ consists of two collections $\set{l_n}{n\geq 2}$ and $\set{\theta_n}{n\geq 1}$ of skew symmetric operations
\begin{gather*}
l_n:A^{\ot n}\to A,\quad \dg{l_n}=2-n, \\
\theta_n:A^{\ot n}\to A,\quad \dg{\theta_n}=k-n+1
\end{gather*}
given by $l_n:=\beta(x^n)$, $\theta_1:=(-1)^{k+1}\beta(\phi)$ and $\theta_n:=\beta(\ul{x}^n)$ for $n\geq 2$,
satisfying
\begin{align}
&d\oo_1 l_n-(-1)^{n}\sum_{m=1}^n l_n\oo_m d = \hspace{-1ex} \sum_{i+j=n+1} \hspace{-1ex} (-1)^{j(i-1)} \hspace{-3ex} \sum_{\sigma\in\USh{j,i-1}} \hspace{-3ex} \sgn{\sigma} (l_i\oo_1 l_j)\oo\sigma^{-1}, \ n\geq 2, \label{EQLInftyAndDer} \\
&\begin{multlined}
d\oo_1\theta_n-(-1)^{k-n+1}\sum_{m=1}^n \theta_n\oo_m d = \\
= - \hspace{-1ex} \sum_{i+j=n+1} \hspace{-1ex} (-1)^{k+j(i-1)} \hspace{-3ex} \sum_{\sigma\in\USh{j,i-1}} \hspace{-3ex} \sgn{\sigma} \left( \theta_i\oo_1 l_j + (-1)^{(k+1)i}l_i\oo_1\theta_j\right)\oo\sigma^{-1}, \quad n\geq 1. \nonumber
\end{multlined}
\end{align}
The indices $i,j,l$ are bound to values where both sides have sense.
Recall that for $f:A^{\ot n}\to A$, we define
$$(\sigma\cdot f)(a_1\ot\cdots\ot a_n)=(f\oo\sigma^{-1})(a_1\ot\cdots\ot a_n)=\pm f(a_{\sigma(1)}\ot\cdots\ot a_{\sigma(n)}),$$
where $\pm$ is the Koszul sign.

The above properties of $l_n$'s and $\theta_n$'s are expressed by ``$\{\theta_n\}$ is a strong homotopy derivation of the $L_\infty$ algebra $(A,\{l_n\})$''.

The sign in the formula defining $\theta_1$ is chosen so that the relation for homotopy derivation is shorter - the terms containing $\theta_i$ for $i\geq 2$ and terms containing $\theta_1$ can then be treated equally since $\sigma\cdot(l_n\oo_1\theta_1)=(\sigma\cdot l_n)\oo_{\sigma(1)}\theta_1=\sgn{\sigma} l_n\oo_{\sigma(1)}\theta_1$.
An even more succinct description can be obtained by passing to the suspension $V=\downar A$, as in Subsection \ref{SECSusp} for $A_\infty$ algebras:

\begin{proposition}
An $L_\infty$ algebra structure is equivalently given by a collection of degree one graded symmetric linear maps $l_n:V^{\otimes n}\rightarrow V,n\geq 1,$ that satisfy the relations (higher order Jacobi relations)
$$\sum_{j=1}^n\sum_{\sigma}(-1)^{e(\sigma)}l_{n-j+1}(l_j(v_{\sigma(1)},\dots, v_{\sigma(j)}),v_{\sigma(j+1)},\dots,v_{\sigma(n)})=0$$
where $\sigma$ runs over all $(j, n-j)$ unshuffle permutations.
The exponent $e(\sigma)$ is the sum of the products of the degrees of the elements that are permuted.

A strong homotopy derivation of degree $k$ of an $L_\infty$ algebra $(V,\{l_n\})$ is a collection of degree $k$ graded symmetric linear maps $\theta_q:V^{\otimes q}\rightarrow V,q\geq 1,$ that satisfy the relations
\begin{gather*}
\sum_{j=1}^n\sum_\sigma(-1)^{e(\sigma)}\theta_{n-j+1}(l_j(v_{\sigma(1)},\dots,v_{\sigma(j)}),v_{\sigma(j+1)},\dots,v_{\sigma(n)}) + {} \\
-(-1)^k(-1)^{e(\sigma)}l_{n-j+1}(\theta_j(v_{\sigma(1)},\dots,v_{\sigma(j)}),v_{\sigma(j+1)},\dots,v_{\sigma(n)})=0
\end{gather*}
where $\sigma$ runs over all $(j,n-j)$ unshuffle permutations.
\end{proposition}

\begin{proof}
Define 
$$L_1:=-d,\quad L_n:=(-1)^{\binom{n}{2}}l_n,\quad \Theta_1:=\theta_1,\quad \Theta_n:=(-1)^{\binom{n}{2}}\theta_n \quad\textrm{for }n\geq 2,$$
and then suspend \eqref{EQLInftyAndDer} as in the proof of the analogous Proposition \ref{shder} for $A_\infty$ algebras.
You get the desired equation up to renaming of the maps.
Also notice that \emph{skew symmetric} map $l_n:A^{\ot n}\to A$ becomes \emph{symmetric} as a map $V^{\ot n}\to V$; similarly for $\theta_n$.
\end{proof}

Again, we remark that this description of $L_\infty$ algebras differs from but is equivalent to the original definition \cite{LS} in which the maps $l_n$ have degree $2-n$ and are graded skew symmetric with the signs adjusted.
The same notion of strong homotopy derivation of $L_\infty$ algebra appeared, for $k=0$, in \cite{LT} and \cite{T}.

It is well known \cite{LM},\cite{LS} that the structure maps $l_n$'s may be extended to a degree $+1$ coderivation $\boldsymbol{l}$ on the symmetric coalgebra  $S^c(V)$ on $V$, and the $L_\infty$ relations are equivalent to $\boldsymbol{l}^2=0$.  
Similarly:

\begin{proposition} \label{PROPLinftyHomDer}
A degree $k$ strong homotopy derivation $\{\theta_q\}$ of a $L_\infty$ algebra $(V,\{l_n\})$ is equivalently described by a degree $k$ coderivation $\boldsymbol{\theta}:S^c(V)\to S^c(V)$ satisfying
$[\boldsymbol{l},\boldsymbol{\theta}]=0$. \qed
\end{proposition}
See \cite{T} for details.

An obvious analogue of Example \ref{EXADoubleBracket} provides examples of strong homotopy derivations of $L_\infty$ algebras.

Following Subsection \ref{SECInner}, we define strong homotopy inner derivations of $L_\infty$ algebras.
A Lie algebra on $A$ becomes, on $V$, an algebra with a symmetric degree $+1$ bracket and the Jacobi identity has the form
$$[[x,y],z]+(-1)^{|y||z|}[[x,z],y]+(-1)^{|x|(|y|+|z|)}[[y,z],x]=0.$$
A degree $k$ derivation $\theta$ of this algebra is defined by $d(\theta(x))=(-1)^{k}\theta(d(x))$ and
$$\theta[x,y]=(-1)^k[\theta(x),y]+(-1)^k(-1)^{k|x|}[x,\theta(y)].$$
A straightforward calculation yields:

\begin{proposition}
Let $a\in V$ have degree $k$ and satisfy $d(a)=0$.
Then the map
$$\theta(x)=[a,x]$$
is a derivation of $V$ of degree $k+1$, called an \emph{inner} derivation.
\qed
\end{proposition}


The inner derivation generalizes to strong homotopy inner derivation as follows:

\begin{proposition}
Let $(V, \{l_n\})$ be an $L_\infty$ algebra and let $a\in V$ have the property that $l_1(a)=0$ and the degree of $a$ is k.  Then the maps
$$\theta_n(v_1,\dots,v_n)=l_{n+1}(a,v_1,\dots,v_n)$$
define a strong homotopy derivation of degree $k+1$ of $V$.
\qed
\end{proposition}

This is proved either by a direct calculation or on the coalgebra level as in Proposition \ref{PROInner} with obvious modifications to the $L_\infty$ case.



\section{Symmetrization and Composition} \label{SECSymComp}

We recall that there is a well known injective coalgebra map $\chi :S^c(V)\longrightarrow T^c(V)$ given by
$$\chi(v_1,\dots,v_n)=\sum_{\sigma\in\Sigma_n}(-1)^{e(\sigma)}v_{\sigma(1)}\otimes \dots\otimes v_{\sigma(n)}$$
where $(-1)^{e(\sigma)}$ is the Koszul sign.

Suppose that $f:T^c(V)\longrightarrow  V$ is a linear map which extends to a coderivation $\boldsymbol{f}:T^c(V)\longrightarrow T^c(V)$ such that $\pi_1 \circ \boldsymbol{f}=f$, where $\pi_1:T^c(V)\longrightarrow V $ is projection.  Then the linear map $f\circ \chi:S^c(V)\longrightarrow V$ extends to the coderivation $\boldsymbol{f}\boldsymbol{\circ}\boldsymbol{\chi}:S^c(V)\longrightarrow S^c(V)$ and the following diagram commutes (Prop. 5, \cite{L})

$$\begin{tikzpicture}
\matrix (m) [matrix of math nodes, row sep=3em, column sep=2.5em, text height=1.5ex, text depth=0.25ex]
{ S^c(V) & & T^c(V) & & \\
S^c(V) & & T^c(V) & & V \\ };
\path[->] (m-1-1) edge node[above]{$\chi$} (m-1-3);
\path[->] (m-2-1) edge node[left]{$\boldsymbol{f}\boldsymbol{\circ}\boldsymbol{\chi}$} (m-1-1);
\path[->] (m-2-3) edge node[left]{$\boldsymbol{f}$} (m-1-3);
\path[->] (m-1-3) edge node[above right]{$\pi_1$} (m-2-5);
\path[->] (m-2-1) edge node[below]{$\chi$} (m-2-3);
\path[->] (m-2-3) edge node[below]{$f$} (m-2-5);
\end{tikzpicture}$$

The symmetrization of an $A_\infty$ algebra structure may  be described by the commutative diagram
\begin{equation} \label{EQDiagram} \begin{tikzpicture}[baseline=-\the\dimexpr\fontdimen22\textfont2\relax]
\matrix (m) [matrix of math nodes, row sep=3em, column sep=2.5em, text height=1.5ex, text depth=0.25ex]
{ S^c(V) & & T^c(V) & & \\
S^c(V) & & T^c(V) & & V \\ };
\path[->] (m-1-1) edge node[above]{$\chi$} (m-1-3);
\path[->] (m-2-1) edge node[left]{$\boldsymbol{l}$} (m-1-1);
\path[->] (m-2-3) edge node[left]{$\boldsymbol{m}$} (m-1-3);
\path[->] (m-1-3) edge node[above right]{$\pi_1$} (m-2-5);
\path[->] (m-2-1) edge node[below]{$\chi$} (m-2-3);
\path[->] (m-2-3) edge node[below]{$m$} (m-2-5);
\end{tikzpicture}\end{equation}
where $m=\sum m_n:T^c(V)\longrightarrow V$ is the collection of  the $A_\infty$ algebra structure maps, $\boldsymbol{m}$ is the lift of $m$ to a coderivation on $T^c(V)$ with $\boldsymbol{m}^2=0$, and the $L_\infty$ algebra structure  $\boldsymbol{l}$ is the lift of the map $m\circ\chi:S^c(V)\longrightarrow V$ to a coderivation on $S^c(V)$.

We now address the issue of symmetrization of strong homotopy derivations of $A_\infty$ algebras.

\begin{proposition}
Let $\theta = \{\theta_n\}$ denote the the collection of maps giving a strong homotopy derivation of degree $+k$ on the $A_\infty$ algebra $(V,m)$.  Regard $\theta$ as a map $T^c(V)\longrightarrow V$ and lift it to the coderivation $\boldsymbol{\theta}$ on $T^c(V)$.  Then the extension of the map $\theta \circ\chi:S^c(V)\longrightarrow V$ to the coderivation $\boldsymbol{\theta'}$ on $S^c(V)$ is a strong homotopy derivation of degree $+k$ on the $L_\infty$ algebra $V$ with algebra structure given by $m\circ\chi$.
\end{proposition}

\begin{proof}
We claim that $[\boldsymbol{l},\boldsymbol{\theta'}]=0$.
We have the commutative diagram
$$\begin{tikzpicture}
\matrix (m) [matrix of math nodes, row sep=3em, column sep=2.5em, text height=1.5ex, text depth=0.25ex]
{ S^c(V) & & T^c(V) & & \\
S^c(V) & & T^c(V) & & V \\ };
\path[->] (m-1-1) edge node[above]{$\chi$} (m-1-3);
\path[->] (m-2-1) edge node[left]{$\boldsymbol{\theta'}$} (m-1-1);
\path[->] (m-2-3) edge node[left]{$\boldsymbol{\theta}$} (m-1-3);
\path[->] (m-1-3) edge node[above right]{$\pi_1$} (m-2-5);
\path[->] (m-2-1) edge node[below]{$\chi$} (m-2-3);
\path[->] (m-2-3) edge node[below]{$\theta$} (m-2-5);
\end{tikzpicture}$$
and we calculate
$$\chi[\boldsymbol{l},\boldsymbol{\theta'}]=\chi(\boldsymbol{l}\boldsymbol{\theta'}-(-1)^k\boldsymbol{\theta'}\boldsymbol{l})$$
$$=(\chi \boldsymbol{l})\boldsymbol{\theta'}-(-1)^k(\chi \boldsymbol{\theta'})\boldsymbol{l}$$
$$=\boldsymbol{m}(\chi\boldsymbol{\theta'})-(-1)^k\boldsymbol{\theta}(\chi\boldsymbol{l})$$
$$=\boldsymbol{m}\boldsymbol{\theta}\chi-(-1)^k\boldsymbol{\theta}\boldsymbol{m}\chi$$
$$=[\boldsymbol{m},\boldsymbol{\theta}]\chi=0$$
because $\chi\circ\boldsymbol{l}=\boldsymbol{m}\circ\chi$ from the commutative diagram \eqref{EQDiagram} and $[\boldsymbol{m},\boldsymbol{\theta}]=0$ because $\boldsymbol{\theta}$ is a strong homotopy derivation of an $A_\infty$ algebra. Because $\chi$ is injective, it follows that $[\boldsymbol{l},\boldsymbol{\theta'}]=0$.
\end{proof}

\bigskip
The next proposition provides us with a definition for the composition of strong homotopy derivations of $A_\infty$ algebras.
The proof follows immediately from the Jacobi identity.

\begin {proposition}
Let $\boldsymbol{\theta_1}$ and $\boldsymbol{\theta_2}$ be coderivations on $T^c(V)$
of degree $p$ and $q$ respectively.  Suppose that
$$[\boldsymbol{m},\boldsymbol{\theta}_1]=[\boldsymbol{m},\boldsymbol{\theta}_2]=0,$$
i.e. $\btheta_1$ and $\btheta_2$ correspond to strong homotopy derivations of $(V,m)$.  Then $[\btheta_1,\btheta_2]:=\btheta_1\circ\btheta_2-(-1)^{pq}\btheta_2\circ\btheta_1$ is a coderivation of degree $p+q$ with the property 
$$[\bm,[\btheta_1,\btheta_2]]=0,$$
so $[\btheta_1,\btheta_2]$ corresponds to a strong homotopy derivation of $(V,m)$.
\qed
\end{proposition}

There is an obvious analogous proposition for strong homotopy derivations of $L_\infty$ algebras.
A further generalization to $\oP_\infty$ algebras for $\oP$ Koszul (and satisfying certain technical assumptions) follows from Proposition \ref{PROPHomDerAsCoder} in Appendix below.


\section{Appendix: Direct application of the Koszul theory} 

\subsection{The Koszul resolution}

In the first part of Appendix, we prove:

\begin{proposition} \label{PROPCompareoDRwithKoszulII}
Let $G$ be a $\Sigma$ module such that $\Fr{\downar G\op\downar\Phi}$ is of finite type\footnote{For $\Phi$, recall Definition \ref{DEFDkP}; for ``finite type'', see Definition \ref{DEFFT} below.}.
Let $\oP=\Fr{G}/(S)$ be a Koszul operad, where $S\subset\Fr[2]{G}$ are the generating relations.
Then the operad $\oDR{k}$ of Definition \ref{DEFoDR} is isomorphic to the Koszul resolution $\Omega(\Kdual{(\oDP{k})})$ of the operad $\oDP{k}$.
\end{proposition}

Of course, a much more general result has been already obtained in Remark \ref{REMKoszulDiscuss}.
However, here we directly construct the Koszul resolution of $\oDP{k}$.
We hope that some readers find that exercise on Koszul duality interesting.

As usual, instead of computing the Koszul dual cooperad $\Kdual{(\oDP{k})}$ directly, we dualize it to obtain an easier to understand operad.
However, we need to work with operads including arity $1$ operations.
This brings some technical issues with dualizing quadratic data and forces us to assume the finiteness properties.
These assumptions can be avoided by computing $\Kdual{(\oDP{k})}$ directly, but we don't know how to do that.

\begin{remark}
Recall that $k$ is the degree of $\phi\in\Phi$.
The assumption in Proposition \ref{PROPCompareoDRwithKoszulII} that $\Fr{\downar G\op\downar\Phi}$ is of finite type is satisfied if $G$ is of finite type and one of the conditions below is met:
\begin{enumerate}
\item $G(0)=0$, $G(1)$ is concentrated in degree $\geq 2$ (resp. $\leq 0$) and $k\geq 2$ (resp. $k\leq 0$).
\item $G(0), G(1)$ are concentrated in degree $\geq 2$ (resp. $\leq 0$) and $G(\geq 2)$ is concentrated in degree $\geq 1$ (resp. $\leq 1$) and $k\geq 2$ (resp. $k\leq 0$).
\item $G(0)$ is concentrated in degree $\geq 1$ (resp. $\leq 1$), $G(1)$ is concentrated in degree $\geq 2$ (resp. $\leq 0$) and $G(\geq 2)$ is concentrated in degree $\geq 2$ (resp. $\leq 0$) and $k\geq 2$ (resp. $k\leq 0$).
\item $G(0)=\fld\{a\}$, $\dg{a}=0$, $G(2)=\fld\Sigma_2\{b\}$, $\dg{b}=3$, $G(n)=0$ for $n\neq 0,2$ and $k\geq 2$.
\end{enumerate}
These claims follow by an easy tree combinatorics.
The list is, of course, not exhaustive.
We don't know any assumptions which would admit $k=1$.
Curiously, this is the case considered in \cite{KS}.
\end{remark}

In this section, all $\Sigma$ modules and operads have zero differential unless explicitly stated otherwise.
As usual, we consider only conilpotent cooperads, thus the adjective will be omitted in the sequel.
Recall that this assumption implies that the underlying $\Sigma$ module of a cofree cooperad coincides with that of free operad.

We use freely the notation and results of \cite{MSS} and \cite{LV}.

We start by discussing linear duals of (co)operads in some detail.

\begin{definition} \label{DEFFT}
Recall that a $\Sigma$ module $E$ is said to be \emph{of finite type} iff for each $n\geq 0$ $E(n)$ is a dg vector space of finite type, i.e. every degree component $E(n)^d$, $d\in\Z$, is a finite dimensional vector space.

Let $\cSmod$ denote the category of $\Sigma$ modules and let $\cSmodft$ denote the full subcategory of all $\Sigma$ modules of finite type.
There is a subcategory $\Opft$ of $\Op$ of all operads having their underlying $\Sigma$ module in $\cSmodft$.
Similarly for cooperads, there is a full subcategory $\CoOpft$ of $\CoOp$.
\end{definition}

\begin{definition}
Let $E$ be a $\Sigma$ module and let $R\subset\Fr[2]{E}$ be a sub $\Sigma$ module.
Denote this inclusion $i$ and denote $\iota:\Fr{E}\to\Fr[2]{E}/R$ the obvious projection.

Denote $\Op(E,R)$ the operad and $p:\Fr{E}\to\Op(E,R)$ the morphism of operads satisfying $pi=0$ and having the following universal property:
In the diagram below, for every morphism $f$ of operads such that $fi=0$, there exists unique morphism $g$ of operads such that $gp=f$.
\begin{equation*} \begin{tikzpicture}[baseline=-\the\dimexpr\fontdimen22\textfont2\relax]
\matrix (m) [matrix of math nodes, row sep=3em, column sep=2.5em, text height=1.5ex, text depth=0.25ex]
{ R & & \Fr{E} & & \Op(E,R) \\
		& &        & & \oP \\ };
\path[->] (m-1-1) edge node[above]{$i$} (m-1-3);
\path[->] (m-1-3) edge node[above]{$p$} (m-1-5);
\path[->] (m-1-3) edge node[below left]{$\forall f$} (m-2-5);
\path[->,dashed] (m-1-5) edge node[right]{$\exists!g$} (m-2-5);
\end{tikzpicture}\end{equation*}
Denote $\Opft(E,R)$ the operad in $\Opft$ with the same universal property as $\Op(E,R)$, except we consider only all morphisms $f$ with target $\oP\in\Opft$.

Denote $\CoOp(E,R)$ the cooperad and $\pi:\CoOp(E,R)\to\Fr{E}$ the morphism of cooperads satisfying $\iota\pi=0$ and having the following universal property:
In the diagram below, for every morphism $\phi$ of cooperads such that $\iota\phi=0$, there exists unique morphism $\gamma$ of cooperads such that $\pi\gamma=\phi$.
\begin{equation} \label{DIACoOp} \begin{tikzpicture}[baseline=-\the\dimexpr\fontdimen22\textfont2\relax]
\matrix (m) [matrix of math nodes, row sep=3em, column sep=2.5em, text height=1.5ex, text depth=0.25ex]
{ \Fr[2]{E}/R & & \Fr{E} & & \CoOp(E,R) \\
							& &        & & \oC \\ };
\path[->] (m-1-5) edge node[above]{$\pi$} (m-1-3);
\path[->] (m-1-3) edge node[above]{$\iota$} (m-1-1);
\path[->] (m-2-5) edge node[below left]{$\forall\phi$} (m-1-3);
\path[->,dashed] (m-2-5) edge node[right]{$\exists!\gamma$} (m-1-5);
\end{tikzpicture}\end{equation}
Denote $\CoOpft(E,R)$ the cooperad in $\CoOpft$ with the same universal property as $\CoOp(E,R)$, except we consider only all morphisms $f$ with source $\oC\in\CoOpft$.
\end{definition}

\begin{remark} \label{REMUnivPropFinType}
Recall that $\Op(E,R)$ is the usual quotient operad $\Fr{E}/(R)$ of $\Fr{E}$ modulo the ideal $(R)$ generated by $R$.
Observe that if $\Op(E,R)\in\Opft$, then $\Opft(E,R)\cong\Op(E,R)$ by fullness of $\Opft$ in $\Op$.
For cooperads, $\CoOp(E,R)\in\CoOpft$ implies $\CoOpft(E,R)\cong\CoOp(E,R)$ similarly.
\end{remark}

\begin{remark} \label{REMDual}
Here are some well known or easy to prove facts about the linear dual:
\begin{enumerate}
\item The natural map $A^\#\ot B^\#\to(A\ot B)^\#$ is an iso if $A,B$ are of finite type.
\item If $\oC\in\CoOp$ with cocompositions $\Delta_i$, then $\oC^\#\in\Op$ (without any finiteness assumptions) via $\oo_i:\oC(m)^\#\ot\oC(n)^\#\to(\oC(m)\ot\oC(n))^\#\To{\Delta_i^\#}\oC(m+n-1)^\#$.
\item If $\oP\in\CoOpft$, then $\oP^\#\in\CoOpft$.
\item If the free operad (resp. cooperad) $\Fr{E}$ is in $\cSmodft$, then $E\in\cSmodft$ and $\Fr{E}^\#\cong\Fr{E^\#}$ as cooperads (resp. operads).
\item If $\Fr{E}\in\cSmodft$, then $\Op(E,R)\in\Opft$ and $\CoOp(E,R)\in\CoOpft$.
\item If $\oP\in\Opft$ (resp. $\CoOpft$), then there is a natural iso $\oP^{\#\#}\cong\oP$ of operads (resp. cooperads).
\end{enumerate}
\end{remark}

The following lemma is a consequence of an explicit description of $\Fr{E}$:

\begin{lemma} \label{LEMFrTwo}
For every $\sigma(x\oo_{i}y)\in\Fr[2]{E}(n)$ with $x\in E(a), y\in E(b),\ \sigma\in\Sigma_{a+b-1}$, there are unique $\sigma'\in\USh{b,a-1},\ \alpha\in\Sigma_a$ and $\beta\in\Sigma_b$ such that $\sigma(x\oo_{i}y)=\sigma'(\alpha x\oo_{1} \beta y)$.
\qed
\end{lemma}

\begin{lemma} \label{LEMCualizationOfOperad}
If $E$ is a $\Sigma$ module such that $\Fr{E}\in\cSmodft$ and if $R\subset\Fr[2]{E}$ is a sub $\Sigma$ module, then
$$\Op(E,R)^\# \cong \CoOp(E^\#,R^{\perp}) \quad\textrm{and}\quad \CoOp(E,R)^\# \cong \Op(E^\#,R^{\perp})$$
as (co)operads.
The notation $R^\perp$ is explained in the proof below.
\end{lemma}

\begin{proof}
We prove the second statement, the first one is analogous.

Dualize the first row of the diagram \eqref{DIACoOp}:
\begin{equation} \begin{tikzpicture}[baseline=-\the\dimexpr\fontdimen22\textfont2\relax]
\matrix (m) [matrix of math nodes, row sep=3em, column sep=2.5em, text height=1.5ex, text depth=0.25ex]
{ (\Fr[2]{E}/R)^\# & & \Fr{E}^\# & & \CoOp(E,R)^\# \\ };
\path[->] (m-1-3) edge node[above]{$\pi^\#$} (m-1-5);
\path[->] (m-1-1) edge node[above]{$\iota^\#$} (m-1-3);
\end{tikzpicture}\end{equation}
We identify $\Fr{E}^\#\cong\Fr{E^\#}$ by Remark \ref{REMDual}.
In particular, $\Fr[2]{E}^\#\cong\Fr[2]{E^\#}$.
Under this iso, the usual pairing $\Fr[2]{E}^\#\ot\Fr[2]{E}\to\fld$ becomes a pairing $\Psi:\Fr[2]{E^\#}\ot\Fr[2]{E} \to \fld$ described as follows:
For $\alpha\in E(a)^\#,\alpha'\in E(b)^\#,\ e\in E(c),e'\in E(d)$ and $\sigma\in\USh{b,a-1}, \tau\in\USh{d,c-1}$,
\begin{gather} \label{EQPsi}
\Psi\left(\sigma(\alpha\oo_1\alpha')\ot\tau(e\oo_1 e')\right):=
\left\{ \begin{array}{ll}
(-1)^{\dg{\alpha'}\dg{e}}\alpha(e)\cdot\alpha'(e') & \quad \textrm{if } a=c, b=d, \sigma=\tau \\
0																									 & \quad \textrm{else}
\end{array} \right.
\end{gather}
By Lemma \ref{LEMFrTwo}, this determines $\Psi$ uniquely.
We define $R^{\perp}:=\{x\in\Fr[2]{E^\#}\ |\ \forall r\in R \quad \Psi(x\ot r)=0\}$ as usual and then there is the obvious iso $(\Fr[2]{E}/R)^\#\cong R^{\perp}$.
Thus $\iota^\#$ becomes the inclusion $R^\perp\into\Fr{E^\#}$.
Below, we will prove that $\CoOp(E,R)^\#$ has the universal property of $\Opft(E^\#,R^{\perp})$.
By uniqueness and Remarks \ref{REMUnivPropFinType} and \ref{REMDual} ($5.$), the conclusion will follow.

Let $\oP\in\Opft$ and let $f:\Fr{E^\#}\to\oP$ be a morphism of operads such that $fi=0$.
We need to find a unique $g:\CoOp(E,R)^\#\to\oP$ such that the right square commutes:
\begin{equation*} \begin{tikzpicture}[baseline=-\the\dimexpr\fontdimen22\textfont2\relax]
\matrix (m) [matrix of math nodes, row sep=3em, column sep=2.5em, text height=1.5ex, text depth=0.25ex]
{ (\Fr[2]{E}/R)^\# & & \Fr{E}^\# & & \CoOp(E,R)^\# \\
	R^\perp					 & & \Fr{E^\#} & & \oP \\ };
\path[->] (m-1-3) edge node[above]{$\pi^\#$} (m-1-5);
\path[->] (m-1-1) edge node[above]{$\iota^\#$} (m-1-3);
\path[->] (m-2-1) edge node[above]{$i$} (m-2-3);
\path[->] (m-2-3) edge node[above]{$f$} (m-2-5);
\path[-] (m-1-1) edge node[right]{$\cong$} (m-2-1);
\path[-] (m-1-3) edge node[right]{$\cong$} (m-2-3);
\path[->,dashed] (m-1-5) edge node[right]{$g$} (m-2-5);
\end{tikzpicture}\end{equation*}
We dualize this diagram.
By our finiteness assumptions, all operads become cooperads.
By universality of $\CoOp(E,R)$, there is unique $g_1$ such that the diagram below commutes:
\begin{equation*} \begin{tikzpicture}[baseline=-\the\dimexpr\fontdimen22\textfont2\relax]
\matrix (m) [matrix of math nodes, row sep=3em, column sep=2.5em, text height=1.5ex, text depth=0.25ex]
{ (\Fr[2]{E}/R) 			 & & \Fr{E} 			 & & \CoOp(E,R) \\
	(\Fr[2]{E}/R)^{\#\#} & & \Fr{E}^{\#\#} & & \CoOp(E,R)^{\#\#} \\
	(R^\perp)^\#				 & & \Fr{E^\#}^\#  & & \oP^\# \\ };
\path[<-] (m-1-1) edge node[above]{$\iota$} (m-1-3);
\path[<-] (m-1-3) edge node[above]{$\pi$} (m-1-5);
\path[<-] (m-2-3) edge node[above]{$\pi^{\#\#}$} (m-2-5);
\path[<-] (m-2-1) edge node[above]{$\iota^{\#\#}$} (m-2-3);
\path[<-] (m-3-1) edge node[above]{$i^\#$} (m-3-3);
\path[<-] (m-3-3) edge node[above]{$f^\#$} (m-3-5);
\path[-] (m-1-1) edge node[right]{$\cong$} (m-2-1);
\path[-] (m-1-3) edge node[right]{$\cong$} (m-2-3);
\path[-] (m-1-5) edge node[right]{$\cong$} (m-2-5);
\path[-] (m-2-1) edge node[right]{$\cong$} (m-3-1);
\path[-] (m-2-3) edge node[right]{$\cong$} (m-3-3);
\path[->] (m-3-5) edge [bend right=70] node[right]{$g_1$} (m-1-5);
\end{tikzpicture}\end{equation*}
Now we dualize once more and define the morphism $g$ to be the composite 
$$\CoOp(E,R)^\#\To{g_1^\#}\oP^{\#\#}\cong\oP.$$
It is then an easy application of naturality to prove $g$ has the required properties.
\end{proof}

Let $\oQ$ be a Koszul operad with presentation $\oQ=\Fr{E}/(R)$ for some $R\subset\Fr[2]{E}$.
Recall that its Koszul dual cooperad has presentation $\Kdual{\oQ}:=\CoOp(\downar E,\downar^{\ot 2}R)$, where $\downar^{\ot 2}R$ is the image of $R$ under the map $\Fr[2]{E}\to\Fr[2]{\downar E}$ sending $\sigma(e\oo_i e')$ to $(-1)^{\dg{e}}\sigma(\downar e\oo_i\downar e')$.
Denote $\ol{\Kdual{\oQ}}$ the coaugmentation coideal.
Recall that Koszul resolution of $\oQ$ is the cobar construction on $\Kdual{\oQ}$:
$$\Omega(\Kdual{\oQ})=\Fr{\upar\overline{\Kdual{\oQ}}}$$
with differential $\dd$ being the operadic derivation given on generators by the composite
\begin{equation*} \begin{tikzpicture}[baseline=-\the\dimexpr\fontdimen22\textfont2\relax]
\matrix (m) [matrix of math nodes, row sep=3em, column sep=1.5em, text height=1.5ex, text depth=0.25ex]
{ \upar\overline{\Kdual{\oQ}} & & \overline{\Kdual{\oQ}} & & \Fr[2]{\overline{\Kdual{\oQ}}} & & \Fr[2]{\upar\overline{\Kdual{\oQ}}} \\ };
\path[->] (m-1-1) edge node[above]{$\downar$} (m-1-3);
\path[->] (m-1-3) edge node[above]{$\overline{\Delta}_{(1)}$} (m-1-5);
\path[->] (m-1-5) edge node[above]{$\upar\ot\upar$} (m-1-7);
\end{tikzpicture}\end{equation*}
where $\ol{\Delta}_{(1)}$ is related to the infinitesimal coproduct $\Delta_{(1)}:\Kdual{\oQ} \To{\Delta} \Kdual{\oQ}\oo\Kdual{\oQ} \To{\textrm{proj}} \Fr[2]{\Kdual{\oQ}}$ by the formula $\ol{\Delta}_{(1)}(x)=\Delta_{(1)}-\id\oo_1 x-\sum_i x\oo_i\id$.
Rather than computing $\Kdual{\oQ}$ directly, we find it easier to dualize it and use Lemma \ref{LEMCualizationOfOperad}.
Then we dualize again to obtain a cooperad isomorphic to $\Kdual{\oQ}$.
These dualizations force us to put the finiteness assumptions on $\oQ$.

\begin{proof}[Proof of Proposition \ref{PROPCompareoDRwithKoszulII}]
$$\oDP{k}=\Fr{G\oplus\Phi}/(S\oplus D),$$
where $D=\{\phi\oo_1 g-(-1)^{k\dg{g}}\sum_i g\oo_i\phi\ |\ g\in G\}$.
We make the Koszul resolution of $\oDP{k}$ explicit.
By definition, 
$$\Kdual{(\oDP{k})}=\CoOp(\downar G\oplus\downar\Phi,\downar^{\ot 2}S\oplus\downar^{\ot 2}D).$$
By assumption, $\Fr{\downar G\oplus\downar\Phi}\in\cSmodft$.
Thus Lemma \ref{LEMCualizationOfOperad} implies
$${\Kdual{(\oDP{k})}}^\# \cong \Op((\downar G)^\#\oplus(\downar\Phi)^\#, (\downar^{\ot 2}S\oplus\downar^{\ot 2}D)^\perp).$$
We compute the orthogonal complement $(\downar^{\ot 2}S\oplus\downar^{\ot 2}D)^\perp$ with respect to $\Psi$ defined at \eqref{EQPsi}:
Denote 
$$(\downar^{\ot 2}S)^{\perp_{G}}:=\{\ol{x}\in\Fr[2]{(\downar G)^\#}\ |\ \forall s\in\downar^{\ot 2}S \quad \Psi(\ol{x},s)=0\}.$$
Since every $\ol{x}\in\Fr[2]{(\downar G)^\#}\subset\Fr[2]{(\downar G)^\#\oplus(\downar\Phi)^\#}$ is trivially orthogonal to $\downar^{\ot 2}D$, we obtain $(\downar^{\ot 2}S)^{\perp_{G}} \subset (\downar^{\ot 2}S\oplus\downar^{\ot 2}D)^\perp$.
Similarly, elements of the form $\ol{\phi}\oo_1\ol{\phi}$, $\ol{\phi}\oo_1\ol{g}$, $\ol{g}\oo_i\ol{\phi}$ are trivially orthogonal to $\downar^{\ot 2}S$, where $\ol{g}\in(\downar G)^\#$ and $\ol{\phi}\in(\downar\Phi)^\#$ is a fixed nonzero element of the $1$-dimensional space.
We immediately see that $\ol{\phi}\oo_1\ol{\phi}$ is also orthogonal to $\downar^{\ot 2}D$.
For the remaining cases, we compute:
\begin{align*}
\Psi(\ol{\phi}\oo_1\ol{g} \ot \downar\phi\oo_1\downar g) &= (-1)^{(1+\dg{g})(1+k)} \ol{\phi}(\downar\phi) \cdot \ol{g}(\downar g), \\
\Psi(\ol{g}\oo_i\ol{\phi} \ot \downar g\oo_j\downar\phi) &= \Psi(\sigma_i(\sigma_i^{-1}\ol{g}\oo_1\ol{\phi}) \ot \sigma_j(\sigma_j^{-1}\downar g\oo_1\downar\phi)) = \\
&= \delta_{ij} (-1)^{(1+k)(1+\dg{g})} (\sigma_i^{-1}\ol{g})(\sigma_j^{-1}\downar g) \cdot \ol{\phi}(\downar\phi) = \\
&= \delta_{ij} (-1)^{(1+k)(1+\dg{g})} \ol{g}(\downar g) \cdot \ol{\phi}(\downar\phi)
\end{align*}
where $\sigma_l\in\USh{1,n-1}$ is the unique unshuffle such that $\sigma_l(1)=l$ and $\delta_{ij}$ is the Kronecker delta.

\begin{equation*}
\begin{matrix}
\begin{tikzpicture}[baseline=-\the\dimexpr\fontdimen22\textfont2\relax,scale=0.5]
\draw (0,-2)--(0,2);
\draw (-2,-2)--(-2,0)--(0,1);
\draw (-1,-2)--(-1,0)--(0,1);
\draw (1,-2)--(1,0)--(0,1);
\draw (2,-2)--(2,0)--(0,1);
\node at (-1.5,-2) {$\scriptstyle{\cdots}$};
\node at (1.5,-2) {$\scriptstyle{\cdots}$};
\filldraw (0,1) circle (3pt);
\node at (0,1) [above right]{$\ol{g}$};
\filldraw (0,-1) circle (3pt);
\node at (0,-1) [right]{$\ol{\phi}$};
\end{tikzpicture} 
& = &
\begin{tikzpicture}[baseline=-\the\dimexpr\fontdimen22\textfont2\relax,scale=0.5]
\draw (-2,-3)--(-1,-2)--(-1,0)--(-2,1)--(0,2)--(0,3);
\draw (-1,-3)--(0,-2)--(0,0)--(-1,1)--(0,2);
\draw (0,-3)--(-2,-2)--(-2,0)--(0,1)--(0,2);
\draw (1,-3)--(1,1)--(0,2);
\draw (2,-3)--(2,1)--(0,2);
\node at (-1.5,-3) {$\scriptstyle{\cdots}$};
\node at (1.5,-3) {$\scriptstyle{\cdots}$};
\filldraw (0,2) circle (3pt);
\node at (0,2) [above right]{$\ol{g}$};
\filldraw (-2,-1) circle (3pt);
\node at (-2,-1) [left]{$\ol{\phi}$};
\node at (2.8,.5) {$\sigma_i^{-1}$};
\node at (2.5,-2.5) {$\sigma_i$};
\end{tikzpicture}
& = &
\begin{tikzpicture}[baseline=-\the\dimexpr\fontdimen22\textfont2\relax,scale=0.5]
\draw (-2,-2.5)--(-1,-1.5)--(-1,.5)--(0,1.5)--(0,2.5);
\draw (-1,-2.5)--(0,-1.5)--(0,.5)--(0,1.5);
\draw (0,-2.5)--(-2,-1.5)--(-2,.5)--(0,1.5);
\draw (1,-2.5)--(1,.5)--(0,1.5);
\draw (2,-2.5)--(2,.5)--(0,1.5);
\node at (-1.5,-2.5) {$\scriptstyle{\cdots}$};
\node at (1.5,-2.5) {$\scriptstyle{\cdots}$};
\filldraw (0,1.5) circle (3pt);
\node at (0,1.5) [above right]{$\sigma_i^{-1}\ol{g}$};
\filldraw (-2,-.5) circle (3pt);
\node at (-2,-.5) [left]{$\ol{\phi}$};
\node at (2.5,-2) {$\sigma_i$};
\end{tikzpicture} \\
\ol{g}\oo_i\ol{\phi} & & = & & \sigma_i(\sigma_i^{-1}\ol{g}\oo_1\ol{\phi})
\end{matrix}
\end{equation*}
The right-to-left inclusion below follows by an easy calculation:
$$(\downar^{\ot 2}S\oplus\downar^{\ot 2}D)^\perp = (\downar^{\ot 2}S)^{\perp_{G}} \oplus \fld\{\ol{\phi}\oo_1\ol{\phi}\} \oplus \fld\!\left\{ \ol{\phi}\oo_1\ol{g}-(-1)^{(1+k)(1+\dg{g})}\ol{g}\oo_i\ol{\phi} \ \big|\ \ol{g}\in(\downar G)^\#,\ i \right\}$$
and the other inclusion is straightforward.

Now the operad ${\Kdual{(\oDP{k})}}^\#$ is easy to understand:
Intuitively, we can vertically exchange $\ol{\phi}$ with $\ol{g}\in(\downar G)^\#$ in a tree (with levels) encoding iterated $\oo_i$ compositions of operations of ${\Kdual{(\oDP{k})}}^\#$:
\begin{equation*} 
\begin{tikzpicture}[baseline=-\the\dimexpr\fontdimen22\textfont2\relax,scale=0.5]
\draw (-1.5,1)--(-.5,2)--(-.5,3);
\draw (-.5,1)--(-.5,2);
\draw (-.5,-3)--(.5,-2)--(.5,1)--(-.5,2);
\draw (.5,-3)--(.5,-2);
\draw (1.5,-3)--(.5,-2);
\node at (-.5,2) [right]{$\ol{g_1}$};
\filldraw (-.5,2) circle (3pt);
\node at (.5,0) [right]{$\ol{\phi}$};
\filldraw (.5,0) circle (3pt);
\node at (.5,-2) [right]{$\ol{g_2}$};
\filldraw (.5,-2) circle (3pt);
\end{tikzpicture} 
= (-1)^{(1+k)(1+\dg{g_2})}
\begin{tikzpicture}[baseline=-\the\dimexpr\fontdimen22\textfont2\relax,scale=0.5]
\draw (-1.5,1)--(-.5,2)--(-.5,3);
\draw (-.5,1)--(-.5,2);
\draw (-.5,-3)--(-.5,-1)--(.5,0)--(.5,1)--(-.5,2);
\draw (.5,-1)--(.5,0);
\draw (1.5,-1)--(.5,0);
\node at (-.5,2) [right]{$\ol{g_1}$};
\filldraw (-.5,2) circle (3pt);
\node at (-.5,-2) [right]{$\ol{\phi}$};
\filldraw (-.5,-2) circle (3pt);
\node at (.5,0) [right]{$\ol{g_2}$};
\filldraw (.5,0) circle (3pt);
\end{tikzpicture}
= (-1)^{(1+k)(1+\dg{g_1})}
\begin{tikzpicture}[baseline=-\the\dimexpr\fontdimen22\textfont2\relax,scale=0.5]
\draw (-1.5,-1)--(-.5,0)--(-.5,1)--(-.5,3);
\draw (-.5,-1)--(-.5,0);
\draw (-.5,-3)--(.5,-2)--(.5,-1)--(-.5,0);
\draw (.5,-3)--(.5,-2);
\draw (1.5,-3)--(.5,-2);
\node at (-.5,0) [right]{$\ol{g_1}$};
\filldraw (-.5,0) circle (3pt);
\node at (-.5,2) [right]{$\ol{\phi}$};
\filldraw (-.5,2) circle (3pt);
\node at (.5,-2) [right]{$\ol{g_2}$};
\filldraw (.5,-2) circle (3pt);
\end{tikzpicture}
=\cdots
\end{equation*}
Also, if two $\ol{\phi}$'s appear one above the other, then the whole composition vanishes:
\begin{equation*} 
\begin{tikzpicture}[baseline=-\the\dimexpr\fontdimen22\textfont2\relax,scale=0.5]
\draw (-1,-1)--(0,0)--(0,3);
\draw (0,-1)--(0,0);
\draw (1,-3)--(1,-1)--(0,0);
\node at (0,2) [right]{$\ol{\phi}$};
\filldraw (0,2) circle (3pt);
\node at (0,0) [right]{$\ol{g}$};
\filldraw (0,0) circle (3pt);
\node at (1,-2) [right]{$\ol{\phi}$};
\filldraw (1,-2) circle (3pt);
\end{tikzpicture} 
= (-1)^{(1+k)(1+\dg{g})}
\begin{tikzpicture}[baseline=-\the\dimexpr\fontdimen22\textfont2\relax,scale=0.5]
\draw (-1,-3)--(0,-2)--(0,3);
\draw (0,-3)--(0,-2);
\draw (1,-3)--(0,-2);
\node at (0,0) [right]{$\ol{\phi}$};
\filldraw (0,0) circle (3pt);
\node at (0,-2) [right]{$\ol{g}$};
\filldraw (0,-2) circle (3pt);
\node at (0,2) [right]{$\ol{\phi}$};
\filldraw (0,2) circle (3pt);
\end{tikzpicture} 
=
0.
\end{equation*}
Hence only trees with at most one occurrence of $\ol{\phi}$ remains and we assume that $\ol{\phi}$ is always at the root.
The operations in $(\downar G)^\#$ are subject to the same relations as in ${\Kdual{\oP}}^\#$.

Thus
\begin{gather} \label{EQIdentif}
{\Kdual{(\oDP{k})}}^\# \cong {\Kdual{\oP}}^\# \oplus \upar^{1-k}{\Kdual{\oP}}^\# \quad \textrm{as }\Sigma\textrm{ modules.}
\end{gather}
The elements of the second summand are of the form $\ol{\phi}\oo_1\ol{p}$ for some $\ol{p}\in{\Kdual{\oP}}^\#$ and we denote them just by $\ol{\phi}\ol{p}$.
The composition involving the $\upar^{1-k}{\Kdual{\oP}}^\#$ summand is as follows:
\begin{align*}
\ol{\phi}\ol{p}_1 \oo_i \ol{\phi}\ol{p}_2 &= 0, \\
\ol{\phi}\ol{p}_1 \oo_i \ol{p}_2 &= \ol{\phi}(\ol{p}_1\oo_i\ol{p}_2), \\
\ol{p}_1 \oo_i \ol{\phi}\ol{p}_2 &= (-1)^{(1+k)\dg{\ol{p}_1}} \ol{\phi}(\ol{p}_1\oo_i\ol{p}_2),
\end{align*}
where $\oo_i$ on the RHS means composition in ${\Kdual{\oP}}^\#$.

Now we consider ${\Kdual{(\oDP{k})}}^{\#\#}$, which is isomorphic to $\Kdual{(\oDP{k})}$.
As in \eqref{EQIdentif}, its elements are $\tilde{p}\in\Kdual{\oP}$ and $\tilde{\phi}\tilde{p}\in\upar^{k-1}\Kdual{\oP}$.
The partial cocomposition $\Delta_{(1)}$ is dual to the partial composition $\gamma_{(1)}:\Fr[2]{{\Kdual{(\oDP{k})}}^\#} \into {\Kdual{(\oDP{k})}}^\# \oo {\Kdual{(\oDP{k})}}^\# \To{\gamma} {\Kdual{(\oDP{k})}}^\#$, which maps $\sigma(\ol{x}\oo_i\ol{y})$ in $\Fr[2]{{\Kdual{(\oDP{k})}}^\#}$ to $\sigma(\ol{x}\oo_i\ol{y})$ in ${\Kdual{(\oDP{k})}}^\#$:
We denote
\begin{gather} \label{EQInfCoprodP}
\Delta_{(1)}(\tilde{p}) = \id\oo_1\tilde{p} + \sum_{i} \tilde{p}\oo_i\id + \sum_{l} \sigma_l \left( \tilde{p}_{1,l} \oo_{i_l} \tilde{p}_{2,l} \right),
\end{gather}
where $\tilde{p}_{1,l},\tilde{p}_{2,l}\in\ol{\Kdual{\oP}}$, and then
\begin{align}
\Delta_{(1)}(\tilde{\phi}\tilde{p}) &= \tilde{\phi}\oo_1\tilde{p} + \sum_{i}(-1)^{(1+k)\dg{\tilde{p}}} \tilde{p}\oo_i\tilde{\phi} + {}\nonumber \\
&\phantom{= } + \sum_{l} \sigma_l \left( \tilde{\phi}\tilde{p}_{1,l} \oo_{i_l} \tilde{p}_{2,l} \right) + \sum_{l} (-1)^{(1+k)\dg{\tilde{p}_{1,l}}} \sigma_l \left( \tilde{p}_{1,l} \oo_{i_l} \tilde{\phi}\tilde{p}_{2,l} \right). \label{EQInfCoprod}
\end{align}
By definition, $\Omega(\Kdual{(\oDP{k})}) = \Fr{\upar\ol{\Kdual{(\oDP{k})}}}$ and hence its generators are $\upar\tilde{p}$ and $\upar\tilde{\phi}\tilde{p}$.
By the above calculation of $\Delta_{(1)}$, the differential of $\Omega(\Kdual{(\oDP{k})})$ is:
\begin{align*}
\dd\left(\upar\tilde{p}\right) &= \sum_{l} (-1)^{\dg{\tilde{p}_{1,l}}} \sigma_l \left( \upar\tilde{p}_{1,l} \oo_i \upar\tilde{p}_{2,l} \right), \\
\dd\left(\upar\tilde{\phi}\tilde{p}\right) &= (-1)^{1+k} \upar\tilde{\phi}\oo_1\upar\tilde{p} + (-1)^{(1+k)\dg{\tilde{p}}+\dg{\tilde{p}}} \sum_{i} \upar\tilde{p}\oo_i\upar\tilde{\phi} + {} \\
&\phantom{= } + \sum_{l} (-1)^{\dg{\tilde{\phi}\tilde{p}_{1,l}}} \sigma_l \left( \upar\tilde{\phi}\tilde{p}_{1,l} \oo_{i_l} \upar\tilde{p}_{2,l} \right) + {} \\
&\phantom{= } + \sum_{l} (-1)^{(1+k)\dg{\tilde{p}_{1,l}}+\dg{\tilde{p}_{1,l}}} \sigma_l \left( \upar\tilde{p}_{1,l} \oo_{i_l} \upar\tilde{\phi}\tilde{p}_{2,l} \right).
\end{align*}
We define an operadic derivation $s:\Fr{\upar\ol{\Kdual{\oP}}}\to\Fr{\upar\ol{\Kdual{(\oDP{k})}}}$ on generators by
$$s(\upar\tilde{p}):=\upar\tilde{\phi}\tilde{p}.$$
Then the formula for $\dd(\upar\tilde{\phi}\tilde{p})$ turns into
\begin{align*}
\dd\left(\upar\tilde{\phi}\tilde{p}\right) &= (-1)^{1+k} \upar\tilde{\phi}\oo_1\upar\tilde{p} + (-1)^{1+k+1+k\dg{p}} \sum_{i} \upar\tilde{p}\oo_i\upar\tilde{\phi} + {} \\
&\phantom{= } + (-1)^{1+k}s(\dd\left(\upar\tilde{p}\right)).
\end{align*}
After comparing with \eqref{DiffOnoDR} for $X:=\upar\ol{\Kdual{\oP}}$, it is obvious that the operad $\Omega(\Kdual{(\oDP{k})})$ is isomorphic to $\oDR{k}$ via $\upar\tilde{\phi}\mapsto(-1)^{1+k}\phi$.
\end{proof}

\subsection{Coderivation description}

In the second part of Appendix, we prove a generalization of Propositions \ref{PROPAinftyHomDerAsCoder} and \ref{PROPLinftyHomDer}:

\begin{proposition} \label{PROPHomDerAsCoder}
Let $G$ be a $\Sigma$ module such that $\Fr{\downar G\op\downar\Phi}$ is of finite type.
Let $\oP$ be a Koszul operad generated by $G$.
Let $\boldsymbol{a}$ be a degree $1$ coderivation of the cofree conilpotent $\Kdual{\oP}$ coalgebra $\Kdual{\oP}\oo A$ satisfying $[\boldsymbol{a},\boldsymbol{a}]=0$ and thus determining a $\oP_\infty$ algebra (including the differential) on a graded vector space $A$.
Then a degree $k$ homotopy derivations of the $\oP_\infty$ algebra are equivalently described as degree $k$ coderivations $\boldsymbol{\theta}$ of the coalgebra $\Kdual{\oP}\oo A$ satisfying
$$[\boldsymbol{a},\boldsymbol{\theta}]=0.$$
\end{proposition}

The technical assumption on generators of $\oP$ is needed to use the results from the previous part of Appendix.

Following $10.1.17$ of \cite{LV}, we now recall the required theory.

Let $\oQ$ be a Koszul operad, let $(A,d)$ be a dg vector space.
Let $\epsilon:\Kdual{\oQ}\to\fld$ be the counit and denote $1\in\Kdual{\oQ}$ the image of $1\in\fld$ under the coaugmentation $\fld\to\Kdual{\oQ}$.
Denote $\Hom_\Sigma(\Kdual{\oQ},\oEnd{A})$ the graded vector space of all $\Sigma$ module morphisms $\Kdual{\oQ}\to\oEnd{A}$ which needn't commute with differentials.
Recall that $\oQ_\infty$ algebras on $(A,d)$, i.e. operad morphisms $\Fr{\upar\ol{\Kdual{\oQ}}}\to\oEnd{A}$, are in bijection with elements $\alpha\in\Hom_\Sigma(\Kdual{\oQ},\oEnd{A})$ satisfying
\begin{gather} \label{EQTwistingMor}
\dg{\alpha}=1, \quad \alpha(1)=d, \quad \alpha*\alpha=0.
\end{gather}
Here $*$ is the convolution pre-Lie product given by
\begin{gather} \label{EQDefConv}
\alpha_1*\alpha_2:=\gamma_{(1)}(\alpha_1\oo_{(1)}\alpha_2)\Delta_{(1)}.
\end{gather}
Denote $\CoDer(\Kdual{\oQ}\oo A)$ the graded vector space of all coderivations $\Kdual{\oQ}\oo A\to \Kdual{\oQ}\oo A$ of the cofree conilpotent $\Kdual{\oQ}$ coalgebra $\Kdual{\oQ}\oo A$.
There is an iso
\begin{gather} \label{EQIsoDGOpCoDer}
\Hom_\Sigma(\Kdual{\oQ},\oEnd{A}) \cong \CoDer(\Kdual{\oQ}\oo A).
\end{gather}
The pre-Lie convolution induces a Lie structure on $\Hom_\Sigma(\Kdual{\oQ},\oEnd{A})$.
The associative composition of maps induces a Lie structure on $\CoDer(\Kdual{\oQ}\oo A)$.
\eqref{EQIsoDGOpCoDer} is in fact an iso of graded Lie algebras.
Finally, this iso takes $\alpha$'s in $\Hom_\Sigma(\Kdual{\oQ},\oEnd{A})$ satisfying \eqref{EQTwistingMor} (notice $\alpha*\alpha=\frac{1}{2}[\alpha,\alpha]$) onto $\boldsymbol{\alpha}$'s in $\CoDer(\Kdual{\oQ}\oo A)$ satisfying 
$$\dg{\boldsymbol{\alpha}}=1, \quad (\epsilon\oo\id)\boldsymbol{\alpha}(1)=d, \quad [\boldsymbol{\alpha},\boldsymbol{\alpha}]=0.$$
Thus $\oQ_\infty$ algebras on a graded vector space $A$ are equivalently described by degree $1$ coderivations $\boldsymbol{\alpha}:\Kdual{\oQ}\oo A \to \Kdual{\oQ}\oo A$ satisfying $[\boldsymbol{\alpha},\boldsymbol{\alpha}]=0$.

\begin{proof}[Proof of Proposition \ref{PROPHomDerAsCoder}]
We apply the above for $\oQ=\oDP{k}\cong\Kdual{\oP}\op\upar^{k-1}\Kdual{\oP}$ with the cooperad structure described in \eqref{EQInfCoprodP} and \eqref{EQInfCoprod}, in particular
\begin{gather*}
\Delta_{(1)}(\Kdual{\oP}) \subset \Kdual{\oP} \oo_{(1)} \Kdual{\oP}, \\ 
\Delta_{(1)}(\upar^{k-1}\Kdual{\oP}) \subset \upar^{k-1}\Kdual{\oP} \oo_{(1)} \Kdual{\oP} \ \op \ \Kdual{\oP} \oo_{(1)} \upar^{k-1}\Kdual{\oP}. 
\end{gather*}
Let $\alpha\in\Hom_\Sigma(\Kdual{\oQ},\oEnd{A})$ satisfy $\dg{\alpha}=1$ and 
$$0=\alpha*\alpha.$$
Let $a:\Kdual{\oP}\to\oEnd{A}$ resp. $b:\upar^{k-1}\Kdual{\oP}\to\oEnd{A}$ be restrictions of $\alpha$.
We get 
$$(\alpha*\alpha)|_{\Kdual{\oP}}=a*a,$$
where the convolution product on the right is calculated in the coalgebra $\Kdual{\oP}\oo A$.
Thus $0=[a,a]$. 
Under the iso \eqref{EQIsoDGOpCoDer}, $a$ becomes a degree $1$ coderivation $\boldsymbol{a}$ on $\Kdual{\oP}\oo A$ such that $[\boldsymbol{a},\boldsymbol{a}]=0$.

There is a degree $k-1$ iso $s:\Kdual{\oP}\to\uparrow^{k-1}\Kdual{\oP}$ of $\Sigma$ modules given by $s(\tilde{p})=\tilde{\phi}\tilde{p}$.
Let $\theta:=bs:\Kdual{\oP}\to\oEnd{A}$.
Consider
$$\begin{tikzcd}
\uparrow^{k-1}\Kdual{\oP} \arrow{rr}{\Delta_{(1)}} & & \uparrow^{k-1}\Kdual{\oP}\oo_{(1)}\Kdual{\oP} \ \op \ \Kdual{\oP}\oo_{(1)}\uparrow^{k-1}\Kdual{\oP} \arrow{d}{s^{-1}\oo_{(1)}\id+\id\oo_{(1)}s^{-1}} \\
\Kdual{\oP} \arrow{u}{s} \arrow{r}{\Delta_{(1)}} & \Kdual{\oP}\oo_{(1)}\Kdual{\oP} \arrow{r}{\textrm{diag}} & \Kdual{\oP}\oo_{(1)}\Kdual{\oP} \ \op \ \Kdual{\oP}\oo_{(1)}\Kdual{\oP} \arrow{d}{(-1)^{k-1}\theta\oo_{(1)}a+a\oo_{(1)}\theta} \\
 & \oEnd{A} & \oEnd{A}\oo_{(1)}\oEnd{A} \arrow{l}[swap]{\gamma_{(1)}}
\end{tikzcd}$$
The square commutes by direct calculation using \eqref{EQInfCoprod}.
Going from $\Kdual{\oP}$ to $\oEnd{A}$ both ways, we get
$$(\alpha*\alpha)|_{\uparrow^{k-1}\Kdual{\oP}}\ s = (-1)^{k-1}\theta*a+a*\theta,$$
where the the convolution product on the right is calculated in the coalgebra $\Kdual{\oP}\oo A$.
Thus $0=[a,\theta]$.
Under the iso \eqref{EQIsoDGOpCoDer}, $\theta$ becomes a degree $k$ coderivation $\boldsymbol{\theta}$ on $\Kdual{\oP}\oo A$ such that $0=[\boldsymbol{a},\boldsymbol{\theta}]$.

It is easy to verify that the inverse construction, starting with $\boldsymbol{a},\boldsymbol{\theta}$ and producing $\alpha$, is determined by the same formulas as above.
\end{proof}

\begin{remark}
If we apply Proposition \ref{PROPHomDerAsCoder} for $\oP=\oAss$, then we get coderivations $\boldsymbol{a}$ and $\boldsymbol{\theta}$ on $\Kdual{\oAss}\oo A$ rather than on $T^c(V) \cong \oAss^{\#}\oo\downar A$ as claimed in Proposition \ref{PROPAinftyHomDerAsCoder}.
These two descriptions are related as follows:

From $7.2.2$ of \cite{LV}, we recall that the operadic suspension of an operad $\oE$ is the Hadamard product $\oS\ot\oE$, where $\oS$ is the endomorphism operad on the $1$ dimensional dg vector space $\downar\fld$ concentrated in degree $-1$.
Similarly, there is the operadic suspension $\oS\ot\oC$ of a cooperad $\oC$.
Then it is easy to verify that
$$\Hom_\Sigma(\oC,\oE)\cong\Hom_\Sigma(\oS\ot\oC,\oS\ot\oE)$$
as graded pre-Lie algebras with convolution as in \eqref{EQDefConv}.

In particular, $\Hom_\Sigma(\Kdual{\oP},\oEnd{A}) \cong \Hom_\Sigma(\oS\ot\Kdual{\oP},\oS\ot\oEnd{A})$.
Recall that there is an iso $\oS\ot\oEnd{A} \cong \oEnd{\downar* A}$ of operads.
Finally, recall that the Koszul dual operad of $\oP$ is defined to be $\oP^{!}:=(\oS\ot\Kdual{\oP})^{\#}$ if $\oP$ is of finite type.
Thus 
$$\Hom_\Sigma(\Kdual{\oP},\oEnd{A}) \cong \Hom_\Sigma((\oP^{!})^{\#},\oEnd{\downar* A})$$
as graded Lie algebras.
Hence there is a bijection between degree $1$ coderivations $\boldsymbol{a}$ satisfying $[\boldsymbol{a},\boldsymbol{a}]=0$ on $\Kdual{\oP}\oo A$ and on $(\oP^{!})^{\#}\oo\oEnd{\downar* A}$.
Similarly, there is a degree preserving bijection between coderivations $\boldsymbol{\theta}$ satisfying $[\boldsymbol{a},\boldsymbol{\theta}]=0$ on $\Kdual{\oP}\oo A$ and on $(\oP^{!})^{\#}\oo\oEnd{\downar* A}$.

Applying this to $\oP=\oAss$ resp. $\oLie$ yields Proposition \ref{PROPAinftyHomDerAsCoder} (since $\oAss^{!}=\oAss$) resp. Proposition \ref{PROPLinftyHomDer} (since $\oLie^{!}=\oCom$ and $\oCom^{\#}\oo V \cong S^c(V)$, where $V=\downar A$), at least for $k\neq 1$.
\end{remark}

\begin{example}
Let $\boldsymbol{a}$ be a coderivation as in Proposition \ref{PROPHomDerAsCoder}.
Let $\boldsymbol{\theta}:=\boldsymbol{a}$.
Then $[\boldsymbol{a},\boldsymbol{\theta}]=0$ and we recover Example \ref{EXTautologicalHomDer}.
\end{example}


\end{document}